\documentclass[a4paper,11pt]{amsart}
\usepackage{mathrsfs}
 \usepackage[all]{xy}
\usepackage{amsmath,amssymb,amscd,bbm,amsthm,mathrsfs}
\usepackage{color,xcolor}
\usepackage{graphicx}
\usepackage{manfnt}
\newtheorem{thm}{Theorem}[section]
\newtheorem{lem}{Lemma}[section]
\newtheorem{cor}{Corollary}[section]
\newtheorem{prop}{Proposition}[section]
\newtheorem{rem}{Remark}[section]

\begin{document}
\numberwithin{equation}{section}

 \title[The Laguerre calculus on the nilpotent Lie groups of step two]{The Laguerre calculus on the nilpotent Lie groups of step two
}
\author {Der-Chen Chang, Irina Markina  and Wei Wang
 }
\begin{abstract}
The  Laguerre calculus is widely used for the inversion of differential operators on the Heisenberg group. We extend the Laguerre calculus for nilpotent groups of step two, and test it in the determining of the fundamental solution of the sub-Laplace operator. We also apply  it  to find the Szeg\"o kernels of the projection operators to a kind of regular functions   on the quaternion Heisenberg group.
\end{abstract}

\address{Department of Mathematics and Department of Computer Science,
Georgetown University, Washington D.C. 20057, USA
\newline
Graduate Institute of Applied Science and Engineering, Fu Jen Catholic University, Taipei 242, Taiwan, ROC}
\email{chang@georgetown.edu}
\address{Department of Mathematics, University of Bergen, NO-5008 Bergen, Norway}
\email{irina.markina@uib.no}
\address{Department of Mathematics, Zhejiang University, Zhejiang 310027, PR China}
\email{wwang@zju.edu.cn.}
\thanks{The first author is partially supported by an NSF grant DMS-1408839 and a McDevitt Endowment Fund at Georgetown University. The second author is  partially supported by NFR-DAAD project 267630/F10 and ISP NFR  project 239033/F20.
  The third author is  partially supported by National Nature Science Foundation in China (No.   11571305).
}
 \maketitle

\section{Introduction}
\smallskip
Let us start with a beautiful idea of Mikhlin, contained in his 1936 study of convolution operators on $\mathbb{R}^2$ (see~\cite{Mik}). Let \textbf{F} denote a principal value convolution operator on $\mathbb{R}^2$:
$$
\textbf{F}(\phi)(x) \,=\, \lim_{\varepsilon \to 0} \int_{|y|>\varepsilon}F(y)\phi(x-y)dy,
$$
where $\phi \in C^\infty_0 (\mathbb{R}^2)$ and $F \in C^\infty(\mathbb{R}^2 \setminus \{(0,0)\})$ is homogeneous of degree $-2$ with the vanishing mean value, {\it i.e.,} $F(\lambda y) = \lambda ^{-2}F(y)$ for $\lambda > 0$ and $\int_{|y|=1}F(y)dy=0$. It follows that
$$
F(y)\, =\, \frac{f(\theta)}{r^2},\qquad  y \,=\,  y_1+iy_2 = re^{i\theta},
$$
where $f(\theta) = \sum_{k\in\mathbb{Z}, k\neq 0}f_ke^{ik\theta}$. Suppose that $g$ is another smooth function on $[0,2\pi]$ with $g(\theta) = \sum_{m\in\mathbb{Z}, m\neq 0}g_me^{im\theta}$. Then $g$ induces a principal value convolution operator \textbf{G} on $\mathbb{R}^2$ with kernel $G=\frac{g(\theta)}{r^2}$. In~\cite{Mik}, Mikhlin found the following identity:
  \begin{equation}
\label{eq:Mikhlin}
\frac{|k|i^{-|k|}}{2\pi} \frac{e^{ik\theta}}{r^2}\,\ast\, \frac{|m|i^{-|m|}}{2\pi} \frac{e^{im\theta}}{r^2} \,=\, \frac{|k+m|i^{-|k+m|}}{2\pi} \frac{e^{i(k+m)\theta}}{r^2}.\end{equation}
Here $\ast$ stands for the Euclidean convolution.
Denote the ``symbol" $\sigma(F)$ of $F$ as
$$
\sigma(F)\, =\, \sum_{k\in\mathbb{Z}, k \neq 0}\left (\frac{|k|i^{-|k|}}{2\pi}\right)^{-1} f_ke^{ik\theta}.
$$
With this notation, one may rewrite (\ref{eq:Mikhlin}) as follows:
$$
\sigma(F \ast G) \,=\,  \sigma(F) \cdot \sigma(G).
$$
It is natural to seek a similar calculus in noncommutative setting.
The simplest and most natural noncommutative analogue of the algebra of principal value convolution operators in $\mathbb{R}^n$ is the left-invariant principal value convolution operators on the $n$-dimensional Heisenberg group~$\mathcal{{H}}_n$. Mikhlin's symbol is replaced by a matrix, or tensor, and commutative symbol multiplication becomes noncommutative matrix or tensor multiplication.
This is the so-called Laguerre calculus. Laguerre calculus is the symbolic tensor calculus originally  induced by the Laguerre functions on the Heisenberg group $\mathcal{H}_n$. It was first introduced on $\mathcal{H}_1$ by Greiner~\cite{G} and later extended to
$\mathcal{H}_n$ and $\mathcal{H}_n\times \mathbb{R}^m$ by Beals, Gaveau, Greiner and Vauthier~\cite{BGV, BGGV}.
The Laguerre functions have been used in the study of the twisted convolution, or equivalently, the Heisenberg convolution for several decades.
For example, Geller~\cite{Ge} found a formula that expressed the group Fourier transform of radial functions on the isotropic Heisenberg group, {\it i.e.,} functions $f(z,t)$ that depend only on $|z|^2$ and $t$  in terms of Laguerre transform, and Peetre~\cite{Pe} derived the relation between the Weyl transform and Laguerre calculus.
The connection between Laguerre functions and Fourier analysis on the isotropic $\mathcal{H}_n$ has been exploited in the study of various translation-invariant operators on $\mathcal {H}_n$ by Folland-Stein~\cite{FS}, Jerison~\cite{J}, de Michele-Mauceri~\cite{MM} and Nachman~\cite{N}.  See also Tie \cite{T2},   Chang-Chang-Tie \cite{CCT}, Chang-Greiner-Tie \cite{CGT} for the application to find the inversion of differential operators, and Chang-Tie \cite{CT},  Strichartz \cite{S} for the study of the associated spectral projection operators.

The present paper is   twofold. The first part of the paper can be considered as a continuation of~\cite{BGV, BGGV, BG, BCE, BCT}.
We shall generalize results obtained on the Heisenberg group to general nilpotent Lie groups of step $2$. The second part contains some applications.

  A connected and  simply connected  nilpotent Lie group $\mathcal{N}$ of step $2$ is the vector space $\mathbb{R}^{{m}}\times\mathbb{R}^r$ with the group multiplication given by
  \begin{equation} \label{eq:multipler} (x,t)\cdot (y ,s )=\left(x+y ,
t +s  + 2  B(x,y )  \right
),\end{equation}
  where
 $B: \mathbb{R}^{m}\times \mathbb{R}^{m}\rightarrow
\mathbb{R}^r$ is a skew-symmetric mapping given by
 \begin{equation}\label{eq:Phi-B}
B(x,y )=\left(B^1(x,y ),\ldots,B^r(x,y )\right),\qquad B^\beta(x,y )=  \sum_{ k,l=1}^{m}    B^\beta_{kl}  x_ky_l.  \end{equation}
For simplicity, we do not consider the degenerate case and assume $m=2n$ in this paper. For any $\tau\in  \mathbb{R}^r \setminus\{0\}$, we consider the bilinear form
\begin{equation*}
  B^\tau( x,y):=\sum_{\beta=1}^r \tau_\beta B^\beta( x,y).
\end{equation*}

We show in Section 2  that for $\tau\in S^{r-1}\setminus E$, where $E\subset  S^{r-1}$ is     of Hausdorff dimension    at most $   r-2 $, there exists an  orthonormal basis  locally normalizing $B^\tau$. The  skew-symmetric bilinear form $B^\tau$ can be written in a normal form with respect to this local orthonormal basis
$\{v^\tau_1,\ldots, v^\tau_{{2n}}\}$, which depends on $\tau$ smoothly,  as
\begin{equation}-B^\tau\left(v^\tau_{2j-1},v^\tau_{2 j}\right)
=B^\tau\left(v^\tau_{2 j},v^\tau_{2j-1}\right )
= \mu_j(\tau),\label{eq:mu-j}
\end{equation}$ j=1,2,\ldots n$,  and $B^\tau(v^\tau_k,v^\tau_l)
=0$ for all other choices of subscripts.
   We  can write $y\in \mathbb{R}^{2n}$ in terms of the basis $\{v^\tau_k\}$ as
\begin{equation}\label{eq:y'-''}
  y=\sum_{k=1}^{2n} y_k^\tau v^\tau_k    \in \mathbb{R}^{2n}
\end{equation} for some $y_1^\tau,\ldots,y_{2n}^\tau\in \mathbb{R}$. We call $(y_1^\tau,\ldots,y_{2n}^\tau)$ the {\it $\tau$-coordinates} of the point $y\in \mathbb{R}^{2n}$.

In Section 3, we discuss the {\it twisted convolution} of two functions  $f,g\in L^1(\mathbb{R}^{{2n}})$, which is defined as
\begin{equation*}\begin{split}
    f*_\tau g( y)&=\int_{\mathbb{R}^{2n}} e^{-2iB^\tau(y,x)} f(y-x) g(x)dx,
\end{split}\end{equation*}for any fixed $\tau\in \mathbb{R}^r $.  For a function $ {F}\in L^1(\mathcal{N})  $, denote by $\widetilde{ {F}}_\tau(y )$
the partial Fourier transformation of $ {F}$ (see (\ref{eq:partial-Fourier}) for definition). It is in $L^2(\mathbb{R}^{{2n}})$ for almost all $\tau$ if $ {F}\in L^1(\mathcal{N}) \cap L^2(\mathcal{N})$.

\begin{prop}\label{prop:convolution} For any $\varphi, \psi\in  L^1(\mathcal{N})$, we have\begin{equation}\label{eq:twisted-convolution}
    \widetilde{\varphi}_\tau*_\tau \widetilde{\psi}_\tau =(\widetilde{  \varphi*\psi } )_\tau.
\end{equation}
\end{prop}

We use $\tau$-coordinates to define Laguerre distributions and establish their properties in Section~\ref{sec:laguerre}. Let
$L_k^{(p)}$ be generalized Laguerre polynomials. It is known \cite{BCT} that
\begin{equation}\label{eq:l-L}
l_k^{(p)}(\sigma):=\left [\frac {\Gamma(k+1)}{\Gamma(k+p+1)} \right]^{\frac 12}L_{k}^{(p)}( \sigma)\sigma^{\frac p2} e^{-\frac \sigma2},
\quad
\end{equation} where $\sigma\in [0,\infty),  k, p\in \mathbb {Z}_{\geq 0} $,  constitute  an orthonormal basis of $L^2([0,\infty),d\sigma)$ for fixed  $p$.
We define the  distributions  $ \mathscr  L_{k}^{(p)}$ on $\mathbb{R}^2 \times \mathbb{R}^r$  via their partial Fourier transformations
\begin{equation}\label{eq:exponential-Laguerre0}
 \widetilde{\mathscr  L}_{k}^{(p)}(y,\tau ): =\frac {2 |\tau|}{\pi} ({\rm sgn}\, p)^p l_k^{(|p|)}(2|\tau| |y|^2)e^{ip\theta},
\end{equation}
where $p\in \mathbb{Z}$,  $ y=(y_1,y_2)\in \mathbb{R}^2, y=y_1+iy_2=|y|e^{i\theta},  \tau\in \mathbb{R}^r  $. Then we can define  the {\it exponential Laguerre distribution} ${ \mathscr L}_{\mathbf{k} }^{(\mathbf{p} )}(y,s)$ on $\mathbb{R}^{2n+r}$ via their partial Fourier transformations
\begin{equation}\label{eq:exponential-Laguerre}
 \widetilde{ \mathscr L}_{\mathbf{k} }^{(\mathbf{p} )}(y,\tau ):=\prod_{j=1}^n \mu_j(\dot{\tau})\widetilde{\mathscr  L}_{k_j}^{(p_j)} \left(\sqrt {\mu_j(\dot{\tau})}{\mathbf y^\tau_j}, \tau\right),
\end{equation}
where $ y\in \mathbb{R}^{2n},\tau\in \mathbb{R}^r, \mathbf{p}=(p_1,\ldots p_n)\in \mathbb{Z}^n  , $ $\mathbf{k}=(k_1,\ldots k_n)\in \mathbb{Z}^n_{\geq 0} $, and
\begin{equation}\label{eq:mu-dot}  \begin{split}\dot{ \tau}&=\frac \tau{|\tau|}\in S^{r-1},\qquad
   \mu_j(\tau)=|\tau|\mu_j(\dot{\tau}),\\
{\mathbf y^{\tau}_j}& = \left ( y_{2j-1}^\tau, y_{2 j}^\tau\right)\in \mathbb{R}^2,\qquad  {j}=1,\ldots,n.
\end{split}\end{equation}

The definition of $
 \widetilde{ \mathscr L}_{\mathbf{k} }^{(\mathbf{p} )}(y,\tau )$ above depends on the choice of  local orthonormal basis  normalizing $B^\tau$, and in that local neighbourhood, it smoothly depends on $y$ and $\tau$.
Note that $
 \widetilde{ \mathscr L}_{\mathbf{k} }^{(\mathbf{p} )}(y,\tau )$  is  only defined for  $\tau\in \mathbb{R}^r $ such that  $B^\tau $ is non-degenerate. In the degenerate case, if $ \mu_j(\tau)=0$ for some $j$, we use ordinary Fourier transformation in the direction spanned by $v^\tau_{2j-1}, v^\tau_{2 j} $. For simplicity, we assume  that  $B^\tau $ is non-degenerate for almost all   $\tau\in \mathbb{R}^r $ in this paper. In this case $
 \widetilde{ \mathscr L}_{\mathbf{k} }^{(\mathbf{p} )}(y,\tau )$ is locally integrable and so ${ \mathscr L}_{\mathbf{k} }^{(\mathbf{p} )} $ as a distribution is well defined.

On the other hand, for any fixed $\tau\in \mathbb{R}^r\setminus \{0\}$  with $B^\tau $  non-degenerate, $  \widetilde{ \mathscr L}_{\mathbf{k} }^{(\mathbf{p} )}(\cdot,\tau) $ for fixed $\tau,\mathbf{k} $ and $ \mathbf{p}  $  is a Schwarz function over $\mathbb{R}^{2n}$, and
$\{ \widetilde{ \mathscr L}_{\mathbf{k} }^{(\mathbf{p} )}(\cdot,\tau)\}_{  \mathbf{p} \in \mathbb{Z}^n  ,  \mathbf{k} \in \mathbb{Z}^n_{\geq 0}}   $ constitute   an orthogonal  basis of $L^2(\mathbb{R}^{2n})$   nicely behaving under the    twisted convolution.
\begin{prop} \label{prop:product} For $\mathbf{k},  \mathbf{p},\mathbf{q}, \mathbf{m}\in \mathbb{Z}_+^n$, we have \begin{equation*}
\label{eq:product}
\widetilde{\mathscr  L}_{(\mathbf{k}\wedge \mathbf{p})-\mathbf{1} }^{(\mathbf{p} -\mathbf{k})}\ast_\tau\widetilde{\mathscr  L}_{(\mathbf{q}\wedge \mathbf{m})-\mathbf{1} }^{(\mathbf{q} -\mathbf{m})}
      =\delta_{\mathbf{k}}^{(\mathbf{q})}\widetilde{\mathscr  L}_{(\mathbf{p}\wedge \mathbf{m})-\mathbf{1} }^{(\mathbf{p} -\mathbf{m})},
         \end{equation*}
    where     $\mathbf{p}\wedge \mathbf{m}-\mathbf{1}:=( \min(k_1,p_1)-1,\dots, \min(k_n,p_n)-1) $ and $\delta_{ \mathbf{{k}}}^{( \mathbf{{q}})} $ is the Kronecker delta function.
         \end{prop}

         If we assume   that  $B^\tau $ is non-degenerate for almost all   $\tau\in \mathbb{R}^r $ and $ {F}\in L^1(\mathcal{N}) \cap L^2(\mathcal{N})$, then for almost all $\tau$, $\widetilde{ {F}}_\tau(y )\in L^2(\mathbb{R}^{2n})$
has the Laguerre expansion
\begin{equation*}
   \widetilde{ {F}}_\tau(y )= \sum_{\mathbf{p},\mathbf{k} \in \mathbb{Z}_+^n }  {F}_{\mathbf{k}}^{\mathbf{p}}(\tau ) \widetilde{ \mathscr L}_{\mathbf{p}\wedge\mathbf{k}-1 }^{(\mathbf{p}-\mathbf{k} )}(y,\tau),
\end{equation*}
 with $\sum_{\mathbf{p},\mathbf{k} \in \mathbb{Z}_+^n }^\infty  |{F}_{\mathbf{k}}^{\mathbf{p}}(\tau )|^2<\infty$,  and the Laguerre  tensor of $ {F}$ is defined as
   \begin{equation*}
      \mathcal{ M}_\tau( F  ):=\left( {F}_{\mathbf{k}}^{\mathbf{p}}(\tau )\right )_{\mathbf{p},\mathbf{k} \in \mathbb{Z}_+^n } .
   \end{equation*}

 The following theorem  is the core of the Laguerre calculus on the nilpotent Lie group  $\mathcal{N}$ of step two.
\begin{thm}\label{thm:Laguerre calculus } Suppose that    $B^\tau $ is non-degenerate for almost all   $\tau\in \mathbb{R}^r $.  For $ {F}, G \in L^1(\mathcal{N})\cap L^2(\mathcal{N})$, we have
\begin{equation*}
 \mathcal{ M}_\tau(   F *  {G} )= \mathcal{ M}_\tau(  {F}  )\cdot  \mathcal{ M}_\tau(  G )
\end{equation*}for   almost all   $\tau\in \mathbb{R}^r $.
\end{thm}

Proposition \ref{prop:convolution} and Theorem \ref{thm:Laguerre calculus }   essentially
give  us homomorphisms of noncommutative algebras:
    \begin{equation*}\begin{split}
      ( L^1(\mathcal{N}), *)& \xrightarrow{\rm hom. } (  L^1(\mathbb{R}^{2n}),
        * {}_\tau)
       \xrightarrow{\rm hom.}\mbox{\rm the algebra of } \infty\times\infty-\mbox{\rm matrices} .
  \end{split}  \end{equation*}

The Laguerre calculus can be viewed as a simplification of the group Fourier transformation in some sense.
For any $\tau\in  \mathbb{R}^r \setminus\{0\}$, there exists an irreducible representation $\pi_\tau$ of $\mathcal{N}$ such that for each element $(y,s)\in \mathcal{N}$, $\pi_\tau(y,s)$ is a unitary operator on $L^2(\mathbb{R}^{ n})$.    A crucial  step to apply the group Fourier transformation  effectively  is    to find   matrix elements
\begin{equation}\label{eq:matrix-element}
   \langle\pi_\tau(y,s) h_{\mathbf{k}}, h_{\mathbf{p}}\rangle,
\end{equation} where $\{h_{\mathbf{k}}\}$ is an orthonormal  basis of $L^2(\mathbb{R}^{ n})$ consisting of Hermitian functions. It can be shown that the matrix elements (\ref{eq:matrix-element}) are exactly   Laguerre distributions by using Wigner transformation formula of Hermitian functions (see e.g.
 \cite{wang5}  for $\mathbf{k}=\mathbf{p}$ and  \cite{PR, SW, ww} for $\mathbf{k}=\mathbf{p}=\mathbf{0}$). Then the multiplicativity of Laguerre  tensors  is a corollary of the following property of representations:
$
   \pi_\tau(   F *  {G} )= \pi_\tau(   F   ) \pi_\tau(    {G} )
$ for $ {F}, G \in L^1 (\mathcal{N})$, as Hilbert-Schimdt  operators on $L^2(\mathbb{R}^{n})$.
See \cite[page 21-22]{BCT} for this fact for the Heisenberg group. So it is a simplification of the group Fourier transformation that we   define Laguerre functions directly and establish their properties, without mention representations. Namely, we skip the step from irreducible representations to matrix elements.

In Section 5 we find the Laguerre  tensors of    left invariant differential operators, and apply them to obtain the fundamental solution for the sub-Laplacian in Section 6. From the definition (\ref{eq:exponential-Laguerre}) of Laguerre distributions, we see that $
 \widetilde{ \mathscr L}_{\mathbf{k} }^{(\mathbf{p} )}(y,\tau )$ becomes degenerate as $\tau$ converges to some degenerate point (i.e. $\mu_j(\tau)\rightarrow 0$ for some $j$). For simplicity, we assume  that $B^\tau $ is non-degenerate for   any $0\neq\tau\in \mathbb{R}^r$ in the application (it can be applied to the general case by analyzing the degeneracy of eigenvalues).

\begin{thm} \label{thm:fundamental} Suppose that  $B^\tau $ is non-degenerate for   any $0\neq\tau\in \mathbb{R}^r$. The fundamental solution to the sub-Laplacian on $\mathcal{N}$ is given by the integral
   \begin{equation*}\begin{split}
\frac{\Gamma(n+r-1)}{\pi^n}   \int_{\mathbb{R}^r}\det\left[\frac  { |B^\tau| }{ \sinh   |B^\tau|    } \right]^{\frac 12}\frac {d\tau} { ( \langle |B^\tau| \coth|B^\tau| y,y\rangle+it\cdot  {\tau})^{n+r-1}  },
    \end{split}
\end{equation*}for $y\neq 0$, where $|B^\tau|:= [(B^\tau)^t  B^\tau]^{\frac 12}$ is a $ 2n\times 2n $ symmetric matrix.
\end{thm} We also use the Laguerre calculus to find the Szeg\"o kernel for $k$-CF  functions on the quaternionic Heisenberg group, which was established in \cite{SW} by using the group Fourier transformation. The proof given here by applying the Laguerre calculus is much more easy and clear.

\section{ The  nilpotent
Lie groups of step two and $\tau$-coordinates}
\subsection{ The  nilpotent
Lie groups of step two}Let $\mathcal{N}$ be a nilpotent Lie group of step two with Lie algebra $\mathfrak n$. A nilpotent Lie  algebra $\mathfrak n$ is {\it of step two}
means that $[\mathfrak n,\mathfrak n]$ is central, i.e. $[\mathfrak n,[\mathfrak n,\mathfrak n]]=\{0\}$. Let $\{T_1,\ldots T_r\}$ be a basis of
$[\mathfrak n,\mathfrak n]$. It can be extended to a (Malcev) basis of $\mathfrak n$, $\{T_1,\ldots T_r, Y_1, \ldots, Y_m\}$, with $\dim\mathfrak n=m+r$. Then there exists real numbers $B_{kl}^\beta$'s such that
\begin{equation*}\label{eq:Bjk-s}
    [Y_k,Y_l]=4\sum_{\beta=1}^r B_{kl}^\beta T_\beta, \qquad [T_\beta,Y_k]=[T_\beta,T_\gamma]=0  .
\end{equation*}
 Recall that for a  connected and  simply connected  nilpotent Lie group, the exponential mapping $\exp$ is an analytic diffeomorphism and the Baker-Campell-Hausdorff formula holds \cite{CG}. For nilpotent Lie group $\mathcal{N}$ of step two, this formula becomes
 \begin{equation}\label{eq:CBH}
    \exp X\cdot\exp Y=\exp\left (X+Y+\frac 12 [X,Y]\right)
 \end{equation}for any $X,Y\in \mathfrak n$.
If we identify the group $\mathcal{N}$ with $\mathbb{R}^{m}\times\mathbb{R}^r$ by identifying the element $\exp (\sum_{k =1}^m y_kY_k+ \sum_{\beta=1}^r t_\beta T_\beta)$ with
the point $(y_1,\ldots y_m,t_1,\ldots, t_r)\in \mathbb{R}^{m}\times\mathbb{R}^r$, the Baker-Campell-Hausdorff formula (\ref{eq:CBH}) implies that the multiplication of the group $\mathcal{N}$ can  be written as (\ref{eq:multipler})-(\ref{eq:Phi-B}).  Conversely, for any given  skew-symmetric mapping $B $, the vector space $\mathbb{R}^{{m}}\times\mathbb{R}^r$ with the multiplication given by (\ref{eq:multipler})-(\ref{eq:Phi-B})  is a {\it nilpotent Lie group $\mathcal{N}$ of step two}. The identity element is $(0,0)$. The skew-symmetry of $B$ implies that the inverse of $(x,t)$ is $(-x,-t)$, and the associativity follows from the bilinearity of $B$.

For any $\tau\in  \mathbb{R}^r \setminus\{0\}$, denote matrix
\begin{equation*}
 B^\tau:= \left (\sum_{\beta=1}^r\tau_\beta B^\beta_{l k}\right)
\end{equation*}which is a  skew-symmetric $m\times m$ matrix related to the skew-symmetric mapping in (\ref{eq:Phi-B}). Since $ iB^\tau$ is Hermitian,   eigenvalues of $  B^\tau$ must be pure imaginary. So when $m$ is odd,   $  B^\tau$ has at least one vanishing eigenvalue. For simplicity, we do not consider this degenerate case and assume $m=2n$ in this paper.
Vector fields
\begin{equation}\label{eq:Y}
    Y_k:=\partial_{y_k}+ 2  \sum_{\beta=1}^{r}\sum_{k=1}^{2n}  B^\beta_{ l   k } y_{  l  }
\partial_{ t_\beta} ,
\end{equation}    are left invariant vector fields on $\mathcal{N} $ related to the multiplication in (\ref{eq:multipler}).

Let $\partial_v$ for $v\in \mathbb{R}^{{2n}}$ be the
 derivative    on $\mathbb{R}^{{2n}}$ along the direction $v$, i.e.
$
 \partial_v =\sum_{k=1}^{{2n}}   v_k\partial_{y_k}
$.   Then,
\begin{equation*}Y_v:=\sum_{k=1}^{2n} v_kY_k=\partial_v+  2   B(y,v)\cdot
\partial_t,
\end{equation*}
is a left invariant vector field on $\mathcal{N} $, where
$
   B(y,v )\cdot
\partial_t:= B^1 (y,v )
\partial_{t_1}+\cdots
 +B^r(y,v )
\partial_{t_r}.
$
Their brackets are
  \begin{equation*} \label{eq:bracket-Phi}\begin{split}
 [Y_v,Y_{v'}]=  4 B(v,v')\cdot
\partial_t.
\end{split}
\end{equation*}

\subsection{Eigenvalues of $  B^{\tau}$ }
  Consider the characteristic polynomial of the matrix $B^\tau$
\begin{equation}\label{eq:Q}
Q(\lambda):=\det (B^{\tau}-\lambda I_{2n})=\sum_{p=1}^{2n}s_p(\tau)\lambda^p.
\end{equation}
The coefficients $s_p(\tau)$ are elements of the polynomial ring $\mathbb R[\tau_1,\ldots,\tau_r]$, that is the ring of polynomials in indeterminate variables $\tau_1,\ldots,\tau_r$ over $\mathbb R$. Since $\mathbb R$ is a field, the polynomial ring $\mathbb R[\tau_1,\ldots,\tau_r]$ is the integral domain and therefore can be extended to the field
\begin{equation*}
   k=\mathbb R(\tau_1,\ldots,\tau_r)
\end{equation*}
  of quotients of $\mathbb R[\tau_1,\ldots,\tau_r]$. In other words, any element in the field $k$ can be represented as a rational function $\frac{f(\tau_1,\ldots,\tau_r)}{h(\tau_1,\ldots,\tau_r)}$,   where polynomials $f,h\neq 0$ belong to $\mathbb R[\tau_1,\ldots,\tau_r]$ (see for instance~\cite[Page 201]{Fral}). Thus the polynomial~\eqref{eq:Q} can be considered as an element of the polynomial ring $k[\lambda]$ over the field $k$. Since every nonconstant polynomial $Q\in k[\lambda]$ can be written as a product of polynomials which are irreducible over the field $k$ (see~\cite[Proposition 2, page 151]{Cox}), we can decompose the polynomial $Q(\lambda)\in k[\lambda]$ into the product
\begin{equation}\label{eq:Q-decomposition}
Q(\lambda)=Q_1^{\alpha_1}(\lambda)\cdot\ldots\cdot Q_p^{\alpha_p}(\lambda)
\end{equation}
of irreducible polynomials over $k$.

We need one more definition, see~\cite[Page 155]{Cox}. Given polynomials $f$, $g\in k[\lambda]$ of positive degrees, we write them in the form
$$ f=a_l\lambda^l+\ldots+a_0,\quad  a_l\neq0,\qquad
g=b_m\lambda^m+\ldots+b_0,\quad b_m\neq0.
$$
The Sylvester matrix of $f$ and $g$ with respect to $\lambda$, denoted by ${\rm Syl}( f, g, \lambda)$ is the coefficient  $  (l + m) \times (l + m) $-matrix:
$$
{\rm Syl}( f, g, \lambda)=\begin{pmatrix}
a_l&   0&0&\ldots&0& b_m&0 &\ldots&0
\\
a_{l-1} &a_l&0& \ldots&0&b_{m-1} &b_m &\ldots&0
\\
a_{l-2}&a_{l-1}&a_l& \ldots&0&b_{m-2}&b_{m-1}&\ddots&0
\\
\vdots&\vdots&\vdots& \ddots&\vdots &\vdots&\vdots&\ddots&\vdots
\\
&&& &a_{l }& && &
\\
\vdots&\vdots&\vdots&\vdots&a_{l-1}&\vdots &\vdots&\vdots&
\\
a_0&&&&\vdots& &&& \vdots
\\
0&a_0&&&&b_0&b_1 &&
\\
\vdots&0&a_0&&&0&b_0&\ddots&\vdots
\\
&\vdots&&\ddots&a_1&\vdots&&\ddots&b_1
\\
0&0&\ldots&0&a_0&0&0&& b_0
\end{pmatrix},
$$
where the empty spaces are filled by zeros and the   coefficients $a_j$ occupies the first  $m$  columns and the coefficients $b_j$ occupies the last   $l$  columns.   The resultant of $f$ and $g$ with respect to $\lambda$, denoted ${\rm Res}( f, g, \lambda)$, is the determinant of the Sylvester matrix:
${\rm Res}(f, g, \lambda) = \det({\rm Syl}( f, g, \lambda))$.

\begin{prop} \cite[Proposition 8, pp. 156]{Cox}\label{prop:common_factor}
Given $f, g\in k[\lambda]$ of positive degree, the resultant ${\rm Res}( f, g, \lambda)\in k$ is an integer polynomial in the coefficients of $f$ and $g$. Furthermore, $f$ and $g$ have a common factor in $k[\lambda]$ if and only if ${\rm Res}( f, g, \lambda) = 0$.
\end{prop}

We write $S^{r-1}$ for the unit sphere in $\mathbb R^r$ and the topology induced from $\mathbb R^r$.  Since for fixed $\tau$, $  B^{\tau}$ is skew-symmetric,  all its eigenvalues  are pure imaginary.

\begin{prop}\label{prop:normal-form} There exists a subset $E$ of $ S^{r-1}$, whose Hausdorff dimension   is at most $   r-2 $,  such that
$  B^{\tau}$ has pure imaginary eigenvalues $i\lambda_1(\tau),  \ldots,  i\lambda_q(\tau)$ of constant multiplicity over $S^{r-1}\setminus E $, that can be ordered as: $\lambda_1(\tau)>\ldots>\lambda_q(\tau)$.
\end{prop}
\begin{proof}
Decompose the polynomial $Q(\lambda)$ into the irreducible ones as in~\eqref{eq:Q-decomposition} and consider one irreducible polynomial $Q_l(\lambda)$.
The common factors of polynomials $Q_l(\lambda)$  and its derivative $Q_l'(\lambda)=\frac{dQ_l(\lambda)}{d\lambda}$ can be detected by the zeros of the resultant ${\rm Res}(Q_l,Q_l'  ,\lambda)$ by Proposition \ref{prop:common_factor}. By definition the resultant ${\rm Res}(Q_l,Q_l' ,\lambda)$, being the determinant of the Sylvester matrix, is a polynomial in coefficients of $Q_l(\lambda)$ and $Q_l'(\lambda) $, thus it an element of $k$.

We need to be careful about the sets in $\mathbb R^r$, where the coefficients of the polynomials $Q_l(\lambda)$ and $ Q_l'(\lambda) $ are not defined. If we write
 \begin{equation*}
    Q_l(\lambda)=\sum_{p=0}^{m_l}(s_l)_p(\tau)\lambda^p,\quad\text{with}\quad (s_l)_p =\frac{(f_l)_p}{(h_l)_p}\in k,
 \end{equation*} for some polynomials $(f_l)_p, (h_l)_p$,
then
\begin{equation*}
   Q_l'(\lambda)= \sum_{p=1}^{m_l}(  s_l)_p(\tau)p\lambda^{p-1} .
\end{equation*}
Recall that a subset $V$ of $\mathbb R^r$ is called {\it real semi-algebraic} if it admits some representation of the form
\begin{equation*}
   V=\bigcup_{i=1}^a\bigcap_{j=1}^{b_i}\left\{x\in\mathbb R^r; P_{i,j}(x)\Box_{ij} 0\right\}
\end{equation*}for some real polynomials $P_{i,j}$, where $\Box_{ij}$ is one of the symbols $\{<,=,>\}$. $V$ is called a {\it real  algebraic set} if each $\Box_{ij} $ is   $=$.
  Then
 \begin{equation*}
 E_l  : =\bigcup_{p}^{m_l}\left\{\tau\in\mathbb R^r:\ (h_l)_p(\tau)=0\right\}
 \end{equation*}
  is a real  algebraic set, and the semi-algebraic set
 \begin{equation*}Z_l:=\left \{\tau\in\mathbb R^r\setminus E_l:\ {\rm Res}\left(Q_l,Q_l',\lambda\right)=0   \right\}
 \end{equation*}
 contains the points in $\mathbb R^r$  where the polynomial $Q_l(\lambda)$ has roots of multiplicity greater or equal to $2$.
There are three options:
 \begin{equation*}
   (1)\, Z_l=\emptyset;\qquad (2)\, \emptyset\neq Z_l\subsetneq \mathbb R^r\setminus E_l;\qquad (3)\, Z_l=\mathbb R^r\setminus E_l.
 \end{equation*}
(1) If $Z_l=\emptyset$, or in other words ${\rm Res}(Q_l,  Q'_l,\lambda)$ is  non-zero on $\mathbb R^r\setminus E_l$, then all the roots of $Q_l(\lambda)$ has constant multiplicity one for any value of $\tau\in \mathbb R^r\setminus E_l$.
(2) If the set $Z_l$ is a non-empty proper subset of $\mathbb R^r\setminus E_l$, then it contains the points $\tau$, where the multiplicity of roots of $Q_l(\lambda)$ is at least 2.   Thus the set $\mathbb R^r\setminus (Z_l\cup E_l)$ is an open set containing points $\tau$, where the multiplicity of any root is equal to one. (3)
The case $Z_l=\mathbb R^r\setminus E_l$ occurs only if ${\rm Res}(Q_l,  Q_l',\lambda)$ is identically zero, but it means that $ Q_l$ and $ Q_l'$ have a common factor, which contradicts to the assumption that $Q_l$ is irreducible. We conclude that for any $\tau  \in \mathbb R^r\setminus E$, where   $E:=Z_l\cup E_l$ is a  real  algebraic set,   the irreducible polynomial $Q_l(\lambda)$ has roots of multiplicity one.
These roots have multiplicity $\alpha_l$ for the polynomial $Q(\lambda)$ due to the decomposition~\eqref{eq:Q-decomposition}.

Repeating the arguments for each of the irreducible polynomials in ~\eqref{eq:Q-decomposition}, we deduce that all of the irreducible polynomials will have simple roots on the set
 \begin{equation}\label{eq:set-E}
    \mathbb R^r\setminus E,\quad\text{with}\quad E:=\bigcup_{l=1}^{ q}(Z_l\cup E_l).
  \end{equation}
Thus the multiplicities of the roots of $Q(\lambda)$ will be locally constant. Recall that a real algebraic  set carries a finite semi-algebraic partition by analytic submanifolds
 of $\mathbb{R}^r$ (cf.   \cite[page 135]{BR}), and so it is of Hausdorff dimension at most $  r-1$.

  The equation (\ref{eq:Q}) is homogeneous in the sense that if
 $(\lambda,\tau)$ is a solution, then  $(s\lambda,s\tau)$
  for  $0\neq s\in \mathbb{R}$ is also a solution of (\ref{eq:Q}) by the trivial property
 of determinants.
So if some eigenvalue of $  B^{\tau}$ is not of constant multiplicity in some neighborhood of $\tau_0$, neither is $ s\tau_0 $
  for any $0\neq s\in \mathbb{R}$.  Namely, $E$ in (\ref{eq:set-E}) is a conic algebraic set. So the intersection $E\cap S^{r-1}$ is an
   algebraic subset of $  S^{r-1}$ of Hausdorff dimension at most $  r-2$.
\end{proof}
\subsection{Normalization of   $B^\tau$ and the $\tau$-coordinates}
Now we can find a smooth orthonormal frame to normalize  $B^\tau$ locally as
Katsumi \cite{Katsumi} did for   symmetric matrices.
\begin{prop} \label{prop:orthonormal-basis} Let $E$ be a subset of $ S^{r-1}$ of Hausdorff dimension     at most $   r-2 $ as in Proposition \ref{prop:normal-form}. Then for any   $\tau_0\in  S^{r-1}\setminus E$, we can find a neighborhood $U$ of $\tau_0\in  S^{r-1}$ and an
   orthonormal basis
$\{v^\tau_1,$ $\ldots, v^\tau_{{2n}}\}$ of $\mathbb{R}^{{2n}}$ smoothly depending on $\tau \in U$,  such that the matrix
$O(\tau)= (v^\tau_1,\ldots, v^\tau_{{2n}})$ normalizes $B^\tau$, i.e.
\begin{equation}\label{eq:J}
O(\tau)^t  B^\tau O(\tau)=J(\tau) :=\left(
  \begin{array}{ccccc}0& -\mu_1(\tau) &0&0&\cdots\\\mu_1(\tau)&0&0&0 &\cdots \\
 0&0&0& - \mu_2(\tau) &\cdots\\0&0&\mu_2(\tau)&0&\cdots\\\vdots&\vdots&\vdots&\vdots& \ddots
   \end{array}
\right),
\end{equation} where $\mu_1(\tau)\geq \mu_2(\tau)\geq\cdots\geq \mu_n(\tau)\geq 0$   also smoothly depend  on $\tau$ in this neighborhood, and $i\mu_1(\tau),-i\mu_1(\tau),\ldots,i\mu_n(\tau),-i\mu_{ n}(\tau)$ represent  repeated pure imaginary eigenvalues of $B^\tau$.
\end{prop}
\begin{proof}
 Let $Q( \lambda ;\tau ) :=\det(B^\tau-\lambda I_{2n})$ be the characteristic polynomial of the matrix $B^\tau$. Write
 \begin{equation*}
   Q( \lambda ;\tau ) = \lambda^{2n} + s_1(\tau) \lambda^{2n-1}+s_2(\tau) \lambda^{k-2}+\cdots + s_{2n}( \tau),
 \end{equation*}
  where the coefficients
$s_1(\tau) ,\cdots ,s_{2n}( \tau)$ are polynomials of $\tau$. By Proposition \ref{prop:normal-form}, $  B^{\tau}$ has pure imaginary eigenvalues $i\lambda_1(\tau),  \cdots,  i\lambda_q(\tau)$ with constant multiplicities $n_1,\ldots, n_q$ and $ \lambda_1(\tau)> \lambda_2(\tau) >\cdots   >\lambda_q(\tau)$. Suppose that  $Q_l$ in (\ref{eq:Q-decomposition}) is of order $k_l$. Then $\frac 1{i^{k_l}}Q_l(i\lambda)$ is a real polynomial with only simple real roots. By applying the implicit function theorem, we see that its
   locally smoothly depend on $\tau$.
So, locally, there exists a polynomial $g(\lambda;\tau )$ in $\lambda$ with coefficients depending
on $\tau$ satisfying
\begin{equation}\label{eq:g}
  g(\lambda;\tau )= ( \lambda - i\lambda_2(\tau ))^{n_1}\ldots   (\lambda -i \lambda_q(\tau )  )^{n_q},
\end{equation}
such that
\begin{equation}\label{eq:Q-g}
    Q ( \lambda ;\tau ) = (\lambda-i\lambda_1(\tau))^{n_1} g(\lambda; \tau)
 \end{equation}
 where $g(i\lambda_1(\tau_0); \tau_0) \neq 0$.

Because the skew symmetric matrix $B^{\tau_0}$ is diagonalizable, there exist
  $n_1$
linearly independent  eigenvectors $V_1,\cdots, V_{n_1}\in\mathbb{C}^r$ with eigenvalue $i\lambda_1(\tau_0)$, i.e.   $B^{\tau_0}V_j=i\lambda_1(\tau_0)V_j$,  $ j=1,\ldots,n_1$. If we set
\begin{equation*}
   Z_j(\tau) = g(B^{\tau }  ;  \tau) V_j,\qquad  j=1,\ldots,n_1,
\end{equation*}
 i.e. $\lambda$ in $ g(\lambda;\tau )$ is replaced by $B^{\tau } $,  then $Z_j(\tau)$'s depend smoothly on $\tau$. Since $ Q (B^{\tau }  ;  \tau) = 0$   by  the well known Cayley-Hamilton theorem,
we have $( B^{\tau } - i\lambda_1(\tau )I)^{n_1} Z_j(\tau)  =0$ by (\ref{eq:Q-g}). Note that $  B^{\tau } - i\lambda_1(\tau )  I_{2n}$ is Hermitian skew symmetric, and so it is diagonalizable. We get
\begin{equation*}
   ( B^{\tau } -i \lambda_1(\tau )  I_{2n}) Z_j(\tau)  =0,
\end{equation*}
i.e. each $ Z_j(\tau)$ is an eigenvector of $B^{\tau }$.
Note that
 \begin{equation*}
    Z_j(\tau_0) =(i \lambda_1(\tau_0) -i \lambda_2(\tau_0))^{n_1}\ldots   \ (i\lambda_1(\tau_0) - i\lambda_q(\tau_0)  )^{n_q}V_j,\qquad  j=1,\ldots,n_1,
 \end{equation*} are linearly independent. It follows that
$Z_1(\tau ), . . ., Z_{n_1}(\tau )$ are linearly independent for every  $\tau$ in a neighborhood of
$\tau_0$. Now we repeat the procedure for $\lambda_2(\tau ) ,\ldots,    \lambda_q(\tau ) $.  Then we  apply the Gram-Schmidt orthogonalization process to
them.

(3) Recall that $i\mu_1(\tau),-i\mu_1(\tau),\ldots,i\mu_n(\tau),-i\mu_{ n}(\tau)$ represent  repeated pure imaginary eigenvalues of $B^\tau$   for real $\tau$. Let $U_j(\tau) +iW_j(\tau) $ be an eigenvector of $B^\tau$ in $\mathbb{C}^{2n}$ with eigenvalue $i\mu_j(\tau)$, i.e.
  $B^\tau(U_j(\tau) +iW_j(\tau) )=i\mu_j(\tau) (U_j(\tau) +iW_j(\tau))$. Since $B^\tau$ is a real matrix for real $\tau$, we see that  $-i \mu_j(\tau) $ is also an eigenvalue of $B^\tau$ with eigenvector $U_j (\tau)-iW_j(\tau) $, and so
\begin{equation}\label{eq:alter}B^\tau U_j(\tau)=-\mu_j(\tau)  W_j(\tau), \qquad
   B^\tau W_j(\tau)=\mu_j(\tau)  U_j(\tau)
\end{equation}for real $\tau$.
Then
\begin{equation*}
   (U_1 +iW_1 ,U_1 -iW_1 ,\ldots,U_n +iW_n , U_{ n} -iW_{ n} )
\end{equation*}
  is a unitary matrix. It follow that
$(U_j+iW_j)^t (\overline{U_k\pm iW_k})=0$ for $j\neq k$, i.e. $ U_j ^t  U_k=0  =W_j ^t  W_k$,  $ W_j ^t  U_k  =0$, and $(U_j+iW_j)^t (\overline{U_j- iW_j})=0$,  $(U_j+iW_j)^t (\overline{U_j+ iW_j})=1$ i.e. $U_j ^t  U_j=\frac 12=  W_j ^t  W_j$, $U_j ^t  W_j  =0$. In summary, the matrix
\begin{equation*}
  O(\tau)=\left(  {\sqrt{2}} W_1(\tau),  {\sqrt{2}}U_1(\tau),\ldots,  {\sqrt{2}} W_n(\tau),  {\sqrt{2}}U_n(\tau)\right)
\end{equation*}
is a $2n\times 2n$ orthogonal matrix. The equations in (\ref{eq:alter})  are equivalent to the equation
 \begin{equation*}
    B^\tau O(\tau)=O(\tau)\left(
  \begin{array}{ccccc}0&-\mu_1 (\tau)&0&0&\cdots\\\mu_1(\tau) &0&0&0 &\cdots \\
 0&0&0& -\mu_2(\tau) &\cdots\\0&0&\mu_2(\tau) &0&\cdots\\\vdots&\vdots&\vdots&\vdots& \ddots
   \end{array}
\right).
 \end{equation*}
 The result follows.
\end{proof}

 Now  in terms of $\tau$-coordinates,
$B^\tau$ can be written as
\begin{equation}\label{eq:B}
   B^\tau(x,y)= \sum_{j =1}^n\mu_j(\tau)\left(-x_{2j-1}^\tau y_{2j }^\tau+x_{2j }^\tau y_{2j-1}^\tau\right),\qquad {\rm for}\quad x, y\in \mathbb{R}^{2n}.
\end{equation}

Recall (see \cite{SW}) that  the $7$-dim quaternionic Heisenberg group is the vector space $\mathbb{R}^4\times \mathbb{R}^3$ with the multiplication given by   (\ref{eq:multipler}) with
\begin{equation}\label{eq:quaternionic-Heisenberg}
\begin{aligned}
B^{1}:&=\left(\begin{array}{cccc} 0 &1 & 0 &0\\ -1& 0& 0& 0\\ 0& 0&0&-1\\0 &0&1 &0\end{array}\right), \   B^{2}:=\left(\begin{array}{cccc} 0 & 0 &
1 &0\\ 0& 0& 0& 1\\ -1& 0&0& 0\\0 &-1& 0 &0\end{array}\right),  \
B^{3}: =\left(\begin{array}{cccc} 0 & 0 & 0 &1\\ 0& 0& -1& 0\\ 0& 1&0& 0\\-1 &0& 0 &0\end{array}\right),
\end{aligned}
\end{equation}
  satisfying the commutating relation of quaternions:
$
   (B^{1})^{2}=(B^{2})^{2}=(B^{3})^{2}=-I_4$, $ B^{1}B^{2}B^{3}=-I_4
. $
Then $B^\tau=\tau_1B^1+\tau_2B^2+\tau_3B^3$ for $\tau\in S^2$ also satisfies
\begin{equation*}
  ( B^\tau)^2=-I_4,
\end{equation*}
  and so its eigenvalues $\mu_1(\tau)\equiv \mu_2(\tau)\equiv 1 $.

\section{The twisted convolution }
For a fixed point $(x,t)\in \mathcal{N}  $, the left multiplication by $(x,t)$ is an affine transformation of  $\mathbb{R}^{{2n+r}}$:
\begin{equation*}
   y\mapsto y+ x  ,\qquad s \mapsto
s  + t +2  B(x,y ),
\end{equation*}
which preserves the Lebegues  measure $dyds$ of $\mathbb{R}^{2n+r}$. The measure $dyds$ is also right invariant, and so it is a Haar measure on the  nilpotent Lie group $\mathcal{N}
$ of step two.
The convolution on $\mathcal{N}  $  is defined as
  \begin{equation*}  \label{eq:convolution}  \varphi
*\psi(y,s):=\int_{\mathcal{N} }\varphi(x ,t )\psi\left((x ,t )^{-1}(y,s)\right) dx dt =\int_{\mathcal{N} }\varphi\left((y,s)(x ,t )^{-1}\right)\psi(x ,t ) dx dt \end{equation*}
for $\phi,\psi\, \in L^1(\mathcal{N} )$.
In general, the algebra $ L^1(\mathcal{N})$ under the convolution is not commutative
\begin{equation*}
   f*g\neq  g*f,\qquad f,g\in  L^1(\mathcal{N}).
\end{equation*}
The {\it
partial Fourier transformation} of a function $\varphi\in  L^1(\mathcal{N})$   is defined as
\begin{equation} \label{eq:partial-Fourier}
   \widetilde{\varphi}_\tau( y)=\int_{\mathbb{R}^r}e^{-i\tau \cdot s}\varphi(y,s)ds, \qquad {\rm for}\quad \tau \in \mathbb{R}^r.
\end{equation}

\begin{prop}\label{prop:continuous} (cf.  \cite[section 4.2]{Mi})
   The
  Fourier transformation and its inverse are continuous on the space $\mathcal{ S}'(\mathbb{R}^{{2n+r}})$ of tempered distributions. So are   the
partial Fourier transformation and its inverse.
\end{prop}

\begin{cor} \label{cor:Minkowski} For $1\leq p\leq\infty$, we have
   \begin{equation}\label{eq:Minkowski}
      \|u*_\tau v\|_{L^p(\mathbb{R}^{{2n}})}\leq \|u \|_{L^1(\mathbb{R}^{{2n}})}\|  v\|_{L^p(\mathbb{R}^{{2n}})}
   \end{equation}
\end{cor}
This is follows from $|u*_\tau v|\leq |u|* | v|$, where $*$ is the Euclidean convolution on $\mathbb{R}^{{2n}}$,  and Minkowski's inequality.

{\it Proof Proposition \ref{prop:convolution}}. Taking partial Fourier transformation on both sides of
 \begin{equation*}\label{eq:convolution-B}
   \varphi*\psi( y,s) =\int_{\mathbb{R}^r}\int_{\mathbb{R}^{2n}}\varphi (y-x, s -t-2  B( y,x))\psi(x,t)dxdt
\end{equation*}
with respect to $s$,
  we get
\begin{equation*}\begin{split}
   (\widetilde{  \varphi*\psi })_\tau(y)&=\int_{\mathbb{R}^r}e^{-i\tau \cdot s}ds\int_{\mathbb{R}^r}\int_{\mathbb{R}^{2n}}\varphi(y-x, s -t- 2  B(y,x))\psi(x, t )dx dt\\&
   =\int_{\mathbb{R}^{2n}}dx\int_{\mathbb{R}^r}\int_{\mathbb{R}^r}e^{-i\tau \cdot [\widetilde{s}+t+ 2  B(y,x)]}\varphi(y-x,\widetilde{s})\psi(x,t)d\widetilde{s}dt\\&
   = \int_{\mathbb{R}^{2n}} e^{-2 i B^{ { \tau}}(y,x)} \widetilde{\varphi}_\tau(y-x) \widetilde{\psi}_\tau(x )dx
\end{split}\end{equation*}
by taking transformation $\widetilde{s}:=s -t- 2  B(y,x)$. Equality
 (\ref{eq:twisted-convolution}) follows.\hskip 44mm $\Box$
\vskip 5mm
  \begin{prop} \label{prop:convol-fourier}
    For $\varphi,\psi\in L^1(\mathcal{N})\cap L^2(\mathcal{N})$, we have
    \begin{equation*}\label{eq:convol-fourier}
       \varphi*\psi(y,t)= \frac 1{(2\pi)^{2n+r}} \int_{\mathbb{R}^{2n+r}}e^{it\cdot \tau+iy\cdot \xi}\widehat{\varphi}\left( T_y(\xi),\tau\right)\widehat{\psi}(\xi,\tau) d\xi\, d\tau,
    \end{equation*}
    where $\widehat{\varphi}$ and $\widehat{\psi}$ is the Euclidean Fourier transformation of $\varphi$ and $\psi$, respectively, and $T_y(\xi)\in \mathbb{R}^{2n }$ with
    \begin{equation*}
     T_y(\xi)_l=  \xi_l-2\sum_{k,\beta} B_{kl}^\beta y_k\tau_\beta.
    \end{equation*}
  \end{prop}  \begin{proof}
  Apply the Euclidean Plancherel formula to the convolution (\ref{eq:convolution-B}) to get
    \begin{equation*}
       \varphi*\psi(y,t)= \frac 1{(2\pi)^{2n+r}} \int_{\mathbb{R}^{2n+r}}\overline{\widehat{\Phi }(\xi,\tau)}\widehat \psi (\xi,\tau) d\xi d\tau,  \end{equation*}
    where $\widehat{\Phi}$ is the Euclidean Fourier transformation of the function
    \begin{equation*}
       \Phi (y',t'):=\overline{\varphi(y-y', t-t' - 2  B(y,y'))},
       \end{equation*}
  for fixed $(y,t)$, {\it i.e.,}
\begin{equation}\label{eq:fourier-cov}\begin{split}
\overline{\widehat{\Phi}(\xi,\tau)}&=\int_{\mathbb{R}^{2n+r}}e^{it'\cdot \tau+iy'\cdot \xi}\varphi(y-y', t-t' - 2  B(y,y'))dy' dt'   \\
&=\int_{\mathbb{R}^{2n+r}}e^{i(t-t''+ 2  B(y,y''))\cdot \tau+i(y-y'')\cdot \xi}\varphi(y'',t'')dy'' dt''   \\
&=e^{i t \cdot \tau+i y \cdot \xi}\int_{\mathbb{R}^{2n+r}}e^{-i t''\tau-i \sum_l\left(\xi_l -2 \sum_{k,\beta} B_{kl}^\beta y_k\tau_ \beta\right) y''_l  }\varphi(y'',t'' )dy'' dt'' \\
&=e^{i t\cdot  \tau+i y \cdot \xi}\widehat{ \varphi}\left(  \ldots,T_y(\xi)_l, \ldots,\tau\right),
\end{split} \end{equation}
by taking the transformation $y''=y-y'$, $t''=t-t' - 2  B(y,y')=t-t' +2 B(y,y'')$, which preserves the volume element.
Here we have used
\begin{equation*}\begin{split}
B(y,y' )\,=&\, B(y,y-y'')\,=\, B(y,y)-B(y,y'')
=  -B(y,y''),
\end{split} \end{equation*} for $y'=y-y''$, which follows from the skew-symmetry of $B$.
  The result follows.
 \end{proof}

\section{The Laguerre basis}~\label{sec:laguerre}
The {\it generalized Laguerre polynomials}
$L_k^{(p)}$ are defined by
  the generating function formula:
\begin{equation}\label{eq:generating-function}
     \sum_{k=0}^\infty    L_{k}^{(p)}( \sigma)z^k=\frac 1{(1-z)^{p+1}}e^{-\frac {\sigma z}{1-z}},\quad \sigma\in \mathbb{R}_+,\quad p\in \mathbb {Z}_{\geq 0},\quad |z|<1.
\end{equation}
(cf.   \cite[section 2.2]{BCT}). In particular, \begin{equation*}
     L_k^{(0)}(\sigma):=\sum_{m=0}^k  { k\choose  m}\frac {(-\sigma)^m}{m!}.
\end{equation*}

The definition $
 \widetilde{ \mathscr L}_{\mathbf{k} }^{(\mathbf{p} )}(y,\tau )$ in (\ref{eq:exponential-Laguerre}) depends on the choice of  local orthonormal basis
$\{v^\tau_1,\ldots, v^\tau_{{2n}}\}$.
By Proposition \ref{prop:orthonormal-basis}, there exists    a subset $E$ of $ S^{r-1}$ of Hausdorff dimension     at most $   r-2 $ and  $  S^{r-1}\setminus E$ can be covered by mutually disjoint Borel subsets $U_1, \ldots, U_N$ such that we can find
   orthonormal basis
$\{v^\tau_1,$ $\ldots, v^\tau_{{2n}}\}$ of $\mathbb{R}^{{2n}}$ normalizing  $B^\tau$, which continuously depend  on $\tau $ in each $ U_j$.
So $
 \widetilde{ \mathscr L}_{\mathbf{k} }^{(\mathbf{p} )}(y,\tau )$ is continuous in each $\mathbb{R}^{{2n}}\times U_j$. We see that it is measurable on $\mathbb{R}^{{2n+r}}$. Moreover, $\widetilde{ \mathscr L}_{\mathbf{k} }^{(\mathbf{p} )}(y,\tau )$ is locally integrable by the following lemma.

\begin{lem}\label{lem:L2}For      $ \tau\in \mathbb{R}^r$ with $B^\tau $   non-degenerate, we have
\begin{equation*}\begin{split}
  \left \|\widetilde{ \mathscr L}_{\mathbf{k} }^{(\mathbf{p} )}(\cdot,\tau )\right\|^2_{L^2(\mathbb{R}^{ 2n })}&= \frac{2^n(\det\left |B^{ {\tau}}\right|)^{\frac 12} }{\pi^n} =\frac{2^n }{\pi^n}\prod_{j=1}^n \mu_j(\tau) , \\
    \left \|\widetilde{ \mathscr L}_{\mathbf{k} }^{(\mathbf{p} )}(\cdot,\tau )\right\|_{L^1(\mathbb{R}^{ 2n })}
    &=  \prod_{j=1}^n \left \|   l_{k_j}^{(p_j)}  \right\|_{L^1(\mathbb{R}^{ 1})} ,
\end{split} \end{equation*}
where $|B^\tau|:= [(B^\tau)^t  B^\tau]^{\frac 12}$.
\end{lem}\begin{proof}
Recall that $\mu_j(\dot{\tau})\neq 0$ for non-degenerate $B^\tau $. Note that  for a fixed $\tau\neq 0$ with $B^\tau $  non-degenerate,    the mapping
\begin{equation}\label{eq:y-y-lambda}
  \mathbb{R}^{2n} \longrightarrow \mathbb{R}^{2n},\qquad y \longmapsto y^\tau:=(y_1^\tau, y_2^\tau,\ldots,y_{2n}^\tau ),
\end{equation} in terms of $\tau$-coordinates in (\ref{eq:y'-''}),  is an orthonormal transformation of $\mathbb{R}^{2n}$. So $dy^\tau=dy $. Then  we have
\begin{equation*}\begin{split}\left\|\widetilde{ \mathscr L}_{\mathbf{k} }^{(\mathbf{p} )}(\cdot,\tau)\right\|_{L^2(\mathbb{R}^{ 2n })}^2& =
  \int_{\mathbb{R}^{ 2n }} dy\prod_{j=1}^n\left |\mu_j(\dot{\tau})\widetilde{\mathscr  L}_{k_j}^{(p_j)} \left( \sqrt {\mu_j(\dot{\tau})}{\mathbf y^\tau_j},\tau \right)\right|^2\\
  &=\int_{\mathbb{R}^{ 2n }} dy\prod_{j=1}^n\left | \mu_j(\dot{\tau})\widetilde{\mathscr  L}_{k_j}^{(p_j)} \left( \sqrt {\mu_j(\dot{\tau})}  {  \mathbf y_j } ,\tau\right)\right|^2\\&=\prod_{j=1}^n \mu_j(\dot{\tau}) \cdot \prod_{j=1}^n \int_{\mathbb{R}^{ 2  }} d{\mathbf y_j}\left | \widetilde{\mathscr  L}_{k_j}^{(p_j)} \left(    {\mathbf y_j } ,\tau\right)\right|^2\end{split}\end{equation*}
    by taking the orthogonal  transformation  $ y \longmapsto y^\tau
$  and  dilation $\sqrt {\mu_j(\dot{\tau})}{ \mathbf y_j}\longrightarrow  {\mathbf y_j} $. On the other hand,
  \begin{equation*}\begin{split} \prod_{j=1}^n \int_{\mathbb{R}^{ 2  }} d{\mathbf y_j}\left | \widetilde{\mathscr  L}_{k_j}^{(p_j)} \left( {  \mathbf y_j } ,\tau\right)\right|^2 &
  =\prod_{j=1}^n \int_0^{2\pi}d\theta\int_0^\infty R dR \left | \frac {2 |\tau|}{\pi} l_{k_j}^{(|p_j|)}(2|\tau| R^2) \right|^2\\&
  =\prod_{j=1}^n\frac { |\tau|}{\pi^2}\int_0^{2\pi}d\theta\int_0^\infty   dR \left |   l_{k_j}^{(|p_j|)}(R) \right|^2=  \frac { |2\tau|^n  }{\pi^n}
  \end{split}\end{equation*} by taking transformation $2 |\tau| R^2\rightarrow R$ and the fact that  $\{l_{k }^{(p )}(\sigma)\}_{k\in \mathbb{Z}_+}$ is an orthonormal basis of $L^2([0,\infty))$ for fixed  $p$. And $|\tau|^n\prod_{j=1}^n \mu_j(\dot{\tau})=(\det\left |B^{ {\tau}}\right|)^{\frac 12}  $ by $B^{ {\tau}}$ in (\ref{eq:J}).  The $L^1$ norm of $\widetilde{ \mathscr L}_{\mathbf{k} }^{(\mathbf{p} )}(\cdot,\tau )$ can be obtained in the same way.
\end{proof}

 We define the   functions  $ \mathscr W_{k}^{(p)}$ on $\mathbb{R}^2\times  \mathbb{R}^1$  via their partial Fourier transformation by
\begin{equation*}
 \widetilde{\mathscr  W}_{k}^{(p)}(x,\tau)=\frac {2 |\tau|}{\pi}({\rm sgn}\, p)^pl_k^{(|p|)}(2|\tau| |x|^2)e^{ip\theta}, \qquad \tau\in \mathbb{R}^1, x\in \mathbb{R}^2.
\end{equation*}This is the usual exponential Laguerre functions on the Heisenberg group $\mathcal{H}_1$. The {\it twisted convolution} $\hat *_\tau$ of two   functions $f,g\in  L^1( \mathbb{R}^2)$  is defined as
\begin{equation} \label{eq:twist-W}
    f\hat *_\tau g( y) =\int_{\mathbb{R}^{2 }} e^{- 2 i\tau (-y_1x_2 +y_2x_1 )} f(y-x) g(x)dx_1dx_2
 \end{equation} for $\tau\in \mathbb{R}^1,y\in \mathbb{R}^2$ (cf. \S 1.2 of \cite{BCT}).
The  twisted convolution of
$ \widetilde{\mathscr W}_{k}^{(p)}$ satisfies the following important  property.
\begin{prop}\label{prop:product-W}{\rm (Theorem 1.3.4 of \cite{BCT})}
\begin{equation*}\label{eq:product-W}
      \widetilde{\mathscr  W}_{ ( {p}\wedge{k})-1 }^{(  {p} - {k})}\hat *_\tau\widetilde{\mathscr  W}_{ ({q}\wedge  {m})-1 }^{( {q} - {m})}
      =\delta_{ {k}}^{( {q})}\widetilde{\mathscr W}_{ ({p}\wedge  {m})-1 }^{( {p} - {m})},
         \end{equation*}where ${k}\wedge  {p}=\min(k,p)$ and $\delta_{ {k}}^{( {q})} $ is the Kronecker delta function.
         \end{prop}

This proposition implies the result of the  twisted convolution of exponential Laguerre functions $\widetilde{\mathscr  L}_{\mathbf{k}  }^{(\mathbf{p}  )}$ in Proposition \ref{prop:product}.

         \vskip 5mm

{\it Proof of Proposition \ref{prop:product}.} For $ \mathbf{{p}}=(p_1,\ldots,p_n)$, $ \mathbf{q}=(q_1,\ldots,q_n) \in \mathbb{Z}^n,$ $ \mathbf{k}=(k_1,\ldots,k_n)$, $ \mathbf{m}=(m_1,\ldots,m_n)\in \mathbb{Z}^n_{\geq 0}  $, we have
       \begin{equation*}\begin{split}
 &\widetilde{\mathscr  L}_{\mathbf{k}  }^{(\mathbf{p}  )}*_\tau \widetilde{\mathscr  L}_{  \mathbf{m} }^{(\mathbf{q} )}(y,\tau)
   =  \int_{\mathbb{R}^{2n}} e^{-2iB^\tau(y,x)} \widetilde{\mathscr  L}_{\mathbf{k}  }^{(\mathbf{p}  )}(y-x)\widetilde{\mathscr  L}_{ \mathbf{m}  }^{(\mathbf{q}  )}(x)dx\\&=
    \int_{\mathbb{R}^{2n}} e^{ -2 i\sum_{j=1}^n \mu_j( {\tau})\left(-y_{2j-1}^{\tau}x_{2j}^{\tau}+y_{2j }^{\tau}x_{2j-1}^{\tau}\right) }\prod_{j=1}^n \mu_j(\dot{\tau})\widetilde{\mathscr  L}_{k_j}^{(p_j)} \left( \sqrt {\mu_j(\dot{\tau})}({\mathbf y^\tau_j} -{\mathbf x^\tau_j}  ),\tau \right)\\&\qquad\qquad\qquad\qquad\qquad\qquad\qquad\qquad\quad\cdot
    \prod_{j=1}^n \mu_j(\dot{\tau})\widetilde{\mathscr  L}_{m_j}^{(q_j)} \left( \sqrt {\mu_j(\dot{\tau})}{\mathbf x^\tau_j },\tau\right) d\mathbf{x^\tau} \\&=\prod_{j=1}^n \mu_j(\dot{\tau})
\int_{\mathbb{R}^{2 }} e^{  -2 i|\tau|\sqrt{ \mu_j( \dot{{\tau}})}\left(-y_{2j-1}^{\tau}x_{2j} +y_{2j }^{\tau}x_{2j-1} \right) }  \widetilde{\mathscr  L}_{k_j}^{(p_j)} \left( \sqrt {\mu_j(\dot{\tau})}{ \mathbf y^\tau_j} -{\mathbf x_j},\tau  \right) \widetilde{\mathscr  L}_{m_j}^{(q_j)} \left( {\mathbf x_j },\tau \right) d{\mathbf x_j}
    \\&
=\prod_{j=1}^n \mu_j(\dot{\tau})\widetilde{\mathscr   W }_{k_j}^{(p_j)}\hat *_{|\tau|} \widetilde{\mathscr   W }_{m_j}^{(q_j)} \left(\sqrt {\mu_j(\dot{\tau})}{\mathbf y^\tau_j},|\tau| \right)  ,
   \end{split}      \end{equation*}
   by using (\ref{eq:B}),  taking the  orthogonal  transformation $x\rightarrow x^\tau $  as in (\ref{eq:y-y-lambda}) and then $\sqrt {\mu_j(\dot{\tau})}{\mathbf x^\tau_j} \rightarrow {\mathbf x_j} $.   Here $\hat *_{|\tau|}$ is the twisted convolution (\ref{eq:twist-W})
 for $|\tau|\in \mathbb{R}$. The result follows from using of Proposition \ref{prop:product-W} for the twisted convolution of
$ \widetilde{\mathscr W}_{k}^{(p)}$. \hskip 70mm $\Box$
 \vskip 5mm

{\it Proof of Theorem \ref{thm:Laguerre calculus }}.
  Recall that $ \{ l_k^{(p)}(\sigma) \}_{   k\in \mathbb{Z}_{\geq 0}} $ in (\ref{eq:l-L}) for any fixed $p$ is an orthonormal basis of $L^2([0,\infty),d\sigma)$, and therefore
 $\{\widetilde{\mathscr  L}_{k}^{(p)}(\cdot,\tau)\}_{ p\in \mathbb{Z} , k\in \mathbb{Z}_{\geq 0}}  $ in (\ref{eq:exponential-Laguerre0}) is an orthogonal   basis of $L^2(\mathbb{R}^2)$ for any fixed $\tau\in \mathbb{R}^r\setminus \{0\}$  with $B^\tau $  non-degenerate. Consequently,
$\{ \widetilde{ \mathscr L}_{\mathbf{k} }^{(\mathbf{p} )}(\cdot,\tau)\}$ in (\ref{eq:exponential-Laguerre}) is an orthogonal      basis of $L^2(\mathbb{R}^{2n})$.

Note that
    for     $ {F}$ and $ {G}$ in  $L^1(\mathcal{N})\cap L^2 (\mathcal{N})$, ${F} *  {G} $ is also in $L^1(\mathcal{N})\cap L^2 (\mathcal{N})$ by  Minkowski's inequality $\|{F} *  {G} \|_{L^p(\mathcal{N})}\leq \|{F}   \|_{L^1(\mathcal{N})}\|  {G} \|_{L^p(\mathcal{N})}$ for $1\leq p\leq+\infty$ (the proof of this inequality works for groups). It is directly to see that for almost all $\tau$, we have $\widetilde{ {F}}_\tau, \widetilde{ {G}}_\tau\in L^1(\mathbb{R}^{2n})\cap L^2 (\mathbb{R}^{2n}) $.

 Note that by   Minkowski's inequality (\ref{eq:Minkowski}) for the twisted convolution, $u*_\tau$ acts continuously on $L^2 (\mathbb{R}^{2n}) $ for $u\in L^1 (\mathbb{R}^{2n}) $.
     Thus,   we have
   \begin{equation*}\begin{split}
   (\widetilde{ {F} *  {G}} )_\tau&= \widetilde{ {F}}_\tau*_\tau \widetilde{ {G}}_\tau=\sum_{|\mathbf{m}| }^\infty \sum_{ |\mathbf{q}|=1}^\infty   {G}_{\mathbf{m}}^{\mathbf{q}}(\tau ) \widetilde{ {F}}_\tau*_\tau\widetilde{ \mathscr L}_{\mathbf{q}\wedge\mathbf{m}-1 }^{(\mathbf{q}-\mathbf{m} )}(\cdot,\tau)\\&= \sum_{|\mathbf{m}| }^\infty \sum_{ |\mathbf{q}|=1}^\infty {G}_{\mathbf{m}}^{\mathbf{q}}(\tau )\left(\sum_{|\mathbf{k}|,|\mathbf{p}|=1}^\infty  {F}_{\mathbf{k}}^{\mathbf{p}}(\tau ) \widetilde{ \mathscr L}_{\mathbf{p}\wedge\mathbf{k}-1 }^{(\mathbf{p}-\mathbf{k} )}(\cdot,\tau)*_\tau \widetilde{ \mathscr L}_{\mathbf{q}\wedge\mathbf{m}-1 }^{(\mathbf{q}-\mathbf{m} )}(\cdot,\tau)\right)\\&=  \sum_{ |\mathbf{m}|   ,|\mathbf{q}|=1}^\infty{G}_{\mathbf{m}}^{\mathbf{q}}(\tau )\left(\sum_{ |\mathbf{p}|=1}^\infty  {F}_{\mathbf{q}}^{\mathbf{p}}(\tau )   \widetilde{ \mathscr L}_{\mathbf{p}\wedge\mathbf{m}-1 }^{(\mathbf{p}-\mathbf{m} )}(\cdot,\tau)\right)
   \end{split} \end{equation*} for almost all $\tau$
   by using Proposition \ref{prop:product}.   Noting that $(\widetilde{ {F} *  {G}} )_\tau\in L^2 (\mathbb{R}^{2n})$, $\sum_{ |\mathbf{m}|   ,|\mathbf{q}|=1}^\infty |{G}_{\mathbf{m}}^{\mathbf{q}}(\tau )|^2 <\infty,$ $\sum_{ |\mathbf{q}|,  |\mathbf{p}|=1}^\infty  |{F}_{\mathbf{q}}^{\mathbf{p}}(\tau ) |^2<  \infty$  and
   \begin{equation*}\begin{split}
       \left\|\sum_{ |\mathbf{p}|=1}^\infty  {F}_{\mathbf{q}}^{\mathbf{p}}(\tau )   \widetilde{ \mathscr L}_{\mathbf{p}\wedge\mathbf{m}-1 }^{(\mathbf{p}-\mathbf{m} )}(\cdot,\tau)\right\|_{L^2 (\mathbb{R}^{2n})}^2&   \leq  \frac{2^n }{\pi^n}\prod_{j=1}^n \mu_j(\tau)\sum_{ |\mathbf{p}|=1}^\infty  |{F}_{\mathbf{q}}^{\mathbf{p}}(\tau ) |^2<  \infty,
   \end{split} \end{equation*}
   we find that
   \begin{equation*}\begin{split}
  \left\langle  (\widetilde{ {F} *  {G}} )_\tau,\widetilde{ \mathscr L}_{\mathbf{p'}\wedge\mathbf{m'}-1 }^{(\mathbf{p'}-\mathbf{m'} )}(\cdot,\tau) \right\rangle_{L^2  }&
   = \sum_{ |\mathbf{m}|   ,|\mathbf{q}|=1}^\infty{G}_{\mathbf{m}}^{\mathbf{q}}(\tau )\left\langle\sum_{ |\mathbf{p}|=1}^\infty  {F}_{\mathbf{q}}^{\mathbf{p}}(\tau )   \widetilde{ \mathscr L}_{\mathbf{p}\wedge\mathbf{m}-1 }^{(\mathbf{p}-\mathbf{m} )}(\cdot,\tau), \widetilde{ \mathscr L}_{\mathbf{p'}\wedge\mathbf{m'}-1 }^{(\mathbf{p'}-\mathbf{m'} )}(\cdot,\tau) \right\rangle_{L^2  }
   \\&
   =\left\|\widetilde{ \mathscr L}_{\mathbf{p'}\wedge\mathbf{m'}-1 }^{(\mathbf{p'}-\mathbf{m'} )}(\cdot,\tau) \right\|_{L^2 (\mathbb{R}^{2n})}^2 \sum_{  |\mathbf{q}|=1}^\infty{G}_{\mathbf{m}'}^{\mathbf{q}}(\tau )   {F}_{\mathbf{q}}^{\mathbf{p}'}(\tau ).
      \end{split} \end{equation*}
 Thus $(\widetilde{ {F} *  {G}} )_\tau $ has the Laguerre expansion
 \begin{equation*}
   (\widetilde{ {F} *  {G}} )_\tau=  \sum_{ |\mathbf{p}|,|\mathbf{m}|    =1}^\infty\left(\sum_{  |\mathbf{q}|=1}^\infty{G}_{\mathbf{m}}^{\mathbf{q}}(\tau )   {F}_{\mathbf{q}}^{\mathbf{p}}(\tau )\right)  \widetilde{ \mathscr L}_{\mathbf{p}\wedge\mathbf{m}-1 }^{(\mathbf{p}-\mathbf{m} )}(\cdot,\tau)  .
 \end{equation*} Here      $ \sum_{ |\mathbf{p}|,|\mathbf{m}|    =1}^\infty\left| \sum_{  |\mathbf{q}|=1}^\infty {G}_{\mathbf{m}}^{\mathbf{q}}(\tau ) {F}_{\mathbf{q}}^{\mathbf{p}}(\tau )\right|^2 $ is   convergent by Cauchy-Schwarz inequality.
    The theorem is proved.   \hskip 132mm $\Box$

\vskip 5mm

Now to recover $F$, we take the inverse partial Fourier transformation
\begin{equation*}
   { {F}}(y,t)=\frac 1{(2\pi)^r}\int_{\mathbb{R}^r} e^{ it\cdot \tau}\widetilde{ {F}}_\tau(y ) d\tau.
\end{equation*}
 In the next section,  we obtain the Laguerre expansion of the kernel of   $\widetilde{\triangle_b}^{-1}$  and then  recover  the kernel of the inverse $\triangle_b^{-1}$ of the sub-Laplacian.

\begin{rem}
On the Heisenberg group $\mathcal{H}_n$, the Laguerre  tensor splits to the positive   and negative parts since the center $\mathbb{R}^1$ has only $2$ directions,
while in the general case, we have to consider each direction represented by a point of the unit sphere $S^{r-1}$ in $\mathbb{R}^r$. At last we integrate
over all directions.
   \end{rem}
   \begin{thm}\label{thm:app} Suppose that $B^\tau $ is non-degenerate for almost all $ \tau\in \mathbb{R}^r$. Then (1) for $f\in L^2(\mathbb{R}^{2n+r})$, we have  $f*\mathscr L_{\mathbf{k} }^{(\mathbf{p} )} \in L^2(\mathbb{R}^{2n+r})$; (2)  for $f\in \mathcal{S}(\mathbb{R}^{2n+r})$, we have
\begin{equation}\label{eq:app}
   \sum_{|\mathbf{k}| =0 }^{+\infty} f* {\mathscr L_{\mathbf{k} }^{(\mathbf{0} )}}     R^{|\mathbf{k}|}\rightarrow f\qquad     {\rm as }\quad R\rightarrow 1-.
\end{equation}
 \end{thm}
\begin{proof} Recall that for a distribution $u\in \mathcal{S}'(\mathbb{R}^m)$ and $\phi \in \mathcal{S} (\mathbb{R}^m)$, we have $\langle \widehat{u}, \phi\rangle=\langle {u},\widehat \phi\rangle $ and $\langle{u},  {\phi}\rangle= \frac 1{(2\pi)^m}\langle\widehat{{u}},\widehat {\phi(-\cdot) }\rangle $, where   $\langle \cdot,\cdot\rangle$ is the dual between $\mathcal{S}'(\mathbb{R}^m)$ and $\mathcal{S}(\mathbb{R}^m)$. Consequently,
  the convolution of  $f \in \mathcal{S} (\mathbb{R}^{2n+r})$ and the  distribution ${\mathscr L_{\mathbf{k} }^{( \mathbf{p}) }}$
 satisfies
 \begin{equation}\label{eq:dual}
     f *{\mathscr L_{\mathbf{k} }^{( \mathbf{p}) }}(y,s):=\left\langle{\mathscr L_{\mathbf{k} }^{( \mathbf{p}) }}, f_{(y,s)}\right\rangle=\frac 1{(2\pi)^r}\left\langle \widetilde {{\mathscr L_{\mathbf{k} }^{( \mathbf{p}) }}},\widetilde{f_1}\right\rangle=\frac 1{(2\pi)^{2n+r}}\left\langle \widehat {{\mathscr L_{\mathbf{k} }^{( \mathbf{p}) }}},\widehat{f_2}\right\rangle
 \end{equation}
 by definition of convolution and the partial Fourier transformation of a distribution, where $f_{(y,s)}  (x ,t )  =f ((y,s)(x ,t )^{-1})$, $f_1(x,t)=f_{(y,s)}  (x ,-t )$  and $f_2(x,t)=f_{(y,s)}  (-x ,-t )$.
Since $\widetilde{{\mathscr L_{\mathbf{k} }^{( \mathbf{p}) }}}$ is locally integrable  and $\widetilde{f_1}\in \mathcal{S}(\mathbb{R}^{2n+r })$, we find that
 \begin{equation*}\begin{split}
    f *{\mathscr L_{\mathbf{k} }^{( \mathbf{p}) }}(y,s)&= \frac { 1}{(2\pi)^r}\int_{\mathbb{R}^{2n+r }}  \widetilde { {\mathscr L_{\mathbf{k} }^{( \mathbf{p}) }} }(x,\tau)\widetilde {f_{1}}\left(x , \tau\right) dx d\tau\\&
    =\frac { 1}{(2\pi)^r}\int_{\mathbb{R}^{ r}}e^{is\cdot \tau  }d\tau  \int_{\mathbb{R}^{2n }} e^{ 2i   B^\tau(y,x) } \widetilde { {\mathscr L_{\mathbf{k} }^{( \mathbf{p}) }} }(x,\tau)\widetilde f_\tau \left( y-x  \right) dx.
 \end{split}\end{equation*}
 Then we get
 \begin{equation*}\begin{split}
    \left\|f *{\mathscr L_{\mathbf{k} }^{( \mathbf{p}) }}\right\|_{L^2(\mathbb{R}^{2n+r})}^2& \leq \int_{\mathbb{R}^{ r}} d\tau  \left\|\widetilde { {\mathscr L_{\mathbf{k} }^{( \mathbf{p}) }} }(\cdot,\tau)\right\|_{L^1(\mathbb{R}^{2n })}  \left\|\widetilde f_\tau \right\|_{L^2(\mathbb{R}^{2n })}^2\\&\leq  \int_{\mathbb{R}^{ r}} d\tau  \prod_{j=1}^n \left \|   l_{k_j}^{(p_j)}  \right\|_{L^1(\mathbb{R}^{ 1})} \left\|\widetilde f_\tau \right\|_{L^2(\mathbb{R}^{2n })}^2=\prod_{j=1}^n \left \|   l_{k_j}^{(p_j)}  \right\|_{L^1(\mathbb{R}^{ 1})} \left\|  f  \right\|_{L^2(\mathbb{R}^{2n+r })}^2
 \end{split}\end{equation*} by   Minkowski's inequality (\ref{eq:Minkowski}) for the twisted convolution and
Lemma \ref{lem:L2}. Consequently,  $f\rightarrow f* {\mathscr L_{\mathbf{k} }^{( \mathbf{p}) }}$ can be extended to a bounded operator from $L^2(\mathbb{R}^{2n+r})$ to itself. Thus    $f*\mathscr L_{\mathbf{k} }^{(\mathbf{p} )}  $ is well defined for $f\in L^2(\mathbb{R}^{2n+r})$ and the above estimate holds.

Note that at the point $0\neq\tau\in \mathbb{R}^r$, we have
\begin{equation*}\label{eq:L-hat}\begin{split}
\widehat{\mathscr L_{\mathbf{k} }^{( \mathbf{0}) }}(\xi,\tau)&= \int_{\mathbb{R}^{2n+r}}e^{-it\cdot \tau-iy\cdot \xi}\mathscr   L_{\mathbf{k} }^{( \mathbf{0}) } (y, t)dtdy= \int_{\mathbb{R}^{2n }}e^{ -iy\cdot  \xi } \widetilde{{\mathscr L}}_{\mathbf{k} }^{ ( \mathbf{0}) } (y ,\tau   )dy \\
&=  \int_{\mathbb{R}^{2n}}e^{ -iy^\tau\cdot\xi^\tau} \widetilde{{\mathscr L}}_{\mathbf{k} }^{ ( \mathbf{0}) } (y  ,\tau  )dy^\tau\\
&= \int_{\mathbb{R}^{2n}}e^{-iy^\tau \cdot \xi^\tau}\prod_{j=1}^n \mu_j(\dot{\tau})\widetilde{\mathscr  L}_{k_j}^{(0)} \left(\sqrt {\mu_j(\dot{\tau})}{\mathbf y^\tau_j}, \tau\right)d{ \mathbf y^\tau_1} \cdots d{\mathbf y^\tau_n}
\\&=\prod_{j=1}^n \int_{\mathbb{R}^{2}}e^{-i{\mathbf y_j^\tau } \cdot{  \Xi_j^\tau} } \mu_j(\dot{\tau})\widetilde{\mathscr  L}_{k_j}^{(0)} \left(\sqrt {\mu_j(\dot{\tau})}{\mathbf y^\tau_j },\tau\right)   d{\mathbf y^\tau_j }
      =\prod_{j=1}^n\widehat{\mathscr  L_{k_j}^{(0)}}\left(\frac {\Xi^\tau_j}{\sqrt { \mu_j(\dot{\tau})}},\tau\right)  ,
  \end{split}\end{equation*}
  by  taking the orthonormal transformation  $ y \longmapsto y^\tau
$ and using definition in (\ref{eq:exponential-Laguerre}),
  where
  \begin{equation*}
     \xi = \xi_1^\tau v_1^\tau+  \cdots+\xi_{2n}^\tau v_{2n}^\tau, \qquad y^\tau \cdot\xi^\tau=y  \cdot\xi \end{equation*}
 (since $\{ v_j^\tau\}$ is an orthonormal basis), and
  \begin{equation}\label{eq:Xi-tau}
     {\Xi^\tau_j}  :=\left(\xi^\tau_{2j-1},\xi^\tau_{2j}\right),\qquad  {j}=1,\ldots,n.
  \end{equation}

Now let us  calculate the Fourier transformation of $\mathscr L_{k}^{(0)}$.
 Note that
\begin{equation}\label{eq:generating-function''}
 \sum_{k=0}^\infty  \widetilde{ \mathscr L}_{k}^{(0)} (x,\tau)z^k=\frac {2|\tau|}{\pi(1-z)}e^{-\frac {1+z}{1-z}|\tau||x|^2},
\end{equation}
since the generating function formula for $l_{k}^{(0)}$ is
\begin{equation}\label{eq:generating-function'}
     \sum_{k=0}^\infty    l_{k}^{(0)}( \sigma)z^k=\frac 1{ 1-z }e^{-\frac {  1+z }{ 1-z }\frac \sigma2},\qquad \sigma\in \mathbb{R}_+,\quad |z|<1,
\end{equation}
by the   generating function formula
(\ref{eq:generating-function}) for $L_{k}^{(0)}$.
Take Fourier transformation with respect to $x\in \mathbb{R}^2$ on both sides of (\ref{eq:generating-function''}) to get
\begin{equation}\label{eq:Fourier-transformation}
 \sum_{k=0}^\infty  \widehat{ \mathscr L_{k}^{(0)}} (\xi,\tau )z^k=\frac {2 }{ 1+z }e^{-\frac {1-z}{1+z}\frac{|\xi|^2}{4|\tau|}}.
\end{equation}Then, applying the   formula
(\ref{eq:generating-function'}) for $z$ replaced by $ -z$ and $\sigma=\frac{|\xi|^2}{2|\tau|}$ at the right hand side (\ref{eq:Fourier-transformation}), we get
\begin{equation}\label{eq:hat-L}
\widehat{ \mathscr L_{k}^{(0)}}(\xi, \tau)=2 (-1)^kl_k^{(0)}\left(\frac {|\xi|^2}{2|\tau|}\right) ,
\end{equation}
for $\xi \in \mathbb{R}^2$. Since $B^\tau $ is non-degenerate for almost all $ \tau\in \mathbb{R}^r$,  we find that
\begin{equation*}
   \widehat{\mathscr L_{\mathbf{k} }^{( \mathbf{0}) }}(\xi,\tau) = \prod_{j=1}^nl_{k_j}^{(0)}\left(\frac {|\Xi^\tau_j |^2}{2\mu_j( {\tau})}\right)\cdot 2 (-1)^{k_j}
\end{equation*}
is uniformly bounded a.e. on $\mathbb{R}^{2n+r}$ by the definition of $l_{k }^{(0)}$ in (\ref{eq:l-L}). Thus, we have
\begin{equation*}\begin{split}
   \sum_{k_j=0}^\infty \widehat{\mathscr  L_{k_j}^{(0)}}\left(\frac {\Xi^\tau_j }{\sqrt { \mu_j(\dot{\tau})}},\tau\right)R^{k_j}&=2 \sum_{k_j=0}^\infty l_{k_j}^{(0)}\left(\frac {|\Xi^\tau_j |^2}{2\mu_j( {\tau})}\right)(-R)^{k_j} = \frac 2{(1+R) }e^{-\frac {1-R}{1+R}\frac {\left|\Xi^\tau_j \right|^2}{4\mu_j( {\tau})}}
  \end{split}\end{equation*}
by using (\ref{eq:hat-L}),  (\ref{eq:mu-dot}), and  the generating function formula
(\ref{eq:generating-function'}).  Note that by the non-degeneracy of
$B^\tau $, we see that
\begin{equation*}\label{eq:hat-L-f}
   \prod_{j=1}^{n}\frac 2{(1+R) }e^{-\frac {1-R}{1+R}\frac {\left|\Xi^\tau_j \right|^2}{4\mu_j( {\tau})}}\longrightarrow 1
\end{equation*}as $R\rightarrow 1-$, uniformly for $(\xi,\tau)$ in any compact set excluding  an arbitrarily  small neighborhood of the degeneracy set of $\mu_j$'s. It is direct to check that
$\widehat{f_2}(\xi,\tau)=e^{it\cdot \tau+iy\cdot \xi}\widehat{f}(\ldots,\xi_l+2\sum_{k,\beta} B_{kl}^\beta y_k\tau_\beta,\ldots\tau) $. Then apply the above result to (\ref{eq:dual}) and change variables  to get
 \begin{equation*}
  \sum_{|\mathbf{k}| =0 }^{+\infty}f *{\mathscr L_{\mathbf{k} }^{(\mathbf{0} )}}(y,s)R^{|\mathbf{k}|}  =\frac 1{(2\pi)^{2n+r}}\int_{\mathbb{R}^{2n+r}}e^{it\cdot \tau+iy\cdot \xi}  \prod_{j=1}^{n}\frac 2{(1+R) }e^{-\frac {1-R}{1+R}\frac {\left|\widehat{\Xi}^\tau_j \right|^2}{4\mu_j( {\tau})}}  \widehat{f}(\xi,\tau) d\xi\, d\tau
 \end{equation*} for $f\in \mathcal{S} (\mathbb{R}^{2n+r})$, where $\widehat{\Xi}^\tau_j$ is obtained from $ {\Xi}^\tau_j$ in (\ref{eq:Xi-tau}) by replacing $\xi_l$  by $T_l(\xi)$.  The result follows.
  \end{proof}

\section{The Laguerre  tensor of  left invariant differential operators}

 For a  differential operator $D$  on the group $\mathcal{N}$, we denote by $\widetilde{D}$ the {\it partial symbol} of $D$ with respect to $\tau\in \mathbb{R}^r$, i.e. $\partial_{s_\beta}$ is replaced by $i\tau_\beta$. Then, we have
 \begin{equation*}
  \widetilde{ Y_{v }}=\partial_{v } +  2 i  B^\tau\left(y,v \right).
\end{equation*}

\begin{prop}\label{prop:vector-tau} (1) For $Y_{a}$, $a=1,\ldots,2n$, given in~\eqref{eq:Y} and $\varphi\in  L^1(\mathcal{N})$ satisfying $Y_a\varphi\in  L^1(\mathcal{N}) $, we have
\begin{equation*}\label{eq:vector-tau'}
   \widetilde{ Y_{v }}\widetilde{\varphi}_\tau=\widetilde{ (Y_{v } \varphi)}_\tau
\end{equation*}

(2)
For $f,g\in    L^1(\mathbb{R}^{2n}) $ satisfying $\partial_a g\in  L^1(\mathcal{N}) $, $a=1,\ldots,2n$,  we have
  \begin{equation*}\label{eq:vector-tau}
    \widetilde{ Y_{v }}(f*_\tau g)=f*_\tau ( \widetilde{ Y_{v }} g).
  \end{equation*}

  (3) For $f,g,h\in    L^1(\mathbb{R}^{2n}) $,  we have $(f*_\tau g)  *_\tau h=f*_\tau( g *_\tau h ) $.
\end{prop}
\begin{proof}  (1) and (3) follows from definitions directly. Note that
\begin{equation*}\begin{split}
   \widetilde{ Y_{v }}  (f*_\tau g)( y)&=   \widetilde{ Y_{v }}\int_{\mathbb{R}^{2n}} e^{ 2iB^\tau(y,x)} f( x) g(y-x)dx\\&
   = \int_{\mathbb{R}^{2n}} e^{ 2iB^\tau(y,x)} f( x)\{\partial_{v }g(y-x)+2i( B^\tau(v,x)+   B^\tau\left(y,v \right)) g(y-x)\}dx\\&
   = \int_{\mathbb{R}^{2n}} e^{ 2iB^\tau(y,x)} f( x)(\widetilde{ Y_{v }} g)(y-x)dx.
\end{split}\end{equation*} (2) is proved.
\end{proof}

  Let $\{v^\tau_1,$ $\ldots, v^\tau_{{2n}}\}$ be an
orthonormal   basis of $\mathbb{R}^{{2n}}$ given by Proposition \ref{prop:orthonormal-basis}, which   smoothly depends on $\tau $ in an open set $U$. Then
\begin{equation*}
  \widetilde{ Y_{v^\tau_j}}  =\frac {\partial}{\partial {y^\tau_j} }+ 2iB^\tau\left(y,v^\tau_{ j } \right)
\end{equation*}
for $j=1,\ldots, 2n$, by (\ref{eq:mu-j})-(\ref{eq:y'-''}). We need to express   Laguerre  tensor of $ \widetilde{ Y_{v^\tau_j}}$ as an  $\infty\times\infty-\mbox{\rm matrix}$, i.e. the matrix element of $ \widetilde{ Y_{v^\tau_j}}$ acting on the orthogonal basis $\{ \widetilde{ \mathscr L}_{\mathbf{k} }^{(\mathbf{p} )}(y,\tau)\}$ of $L^2(\mathbb{R}^{2n})$.
For this purpose, we introduce the {\it complex $\tau$-coordinates}
\begin{equation*}
  {z_j^\tau}: =   y_{2j-1}^\tau+i y_{2 j}^\tau,
\end{equation*}
and complex horizontal vector fields
\begin{equation*}\label{eq:Z-j}
  Z_j^\tau:=\frac 12\left(Y_{v^\tau_{2j-1}}-i Y_{v^\tau_{2 j}}\right).
\end{equation*}
  Then
\begin{equation}\label{eq:Z-partial}
  \widetilde{Z_j^\tau}= \frac {\partial}{\partial {{{z^\tau_j}}} }+i B^\tau\left(y,v^\tau_{2j-1}\right)+B^\tau\left(y,v^\tau_{2j } \right)
 = \frac {\partial}{\partial {{{z^\tau_j}}} }  -\mu_j( {\tau})  \overline{{z}_j^\tau}  ,
 \end{equation}by (\ref{eq:B}),
where
\begin{equation*}
   \frac {\partial}{\partial {{{z^\tau_j}}} }:= \frac 12\left(\frac {\partial}{\partial y_{2j-1}^\tau}-i\frac {\partial}{\partial y_{2j }^\tau }\right).
\end{equation*}
Similarly, set
\begin{equation*}
   \overline{ Z_j^\tau}:=\frac 12\left(Y_{v^\tau_{2j-1}}+i Y_{v^\tau_{2 j}}\right),
\end{equation*}
and
\begin{equation}\label{eq:Z-bar-partial}
  \widetilde{\overline{Z_j^\tau} }
 = \frac {\partial}{\partial  \overline{{{z^\tau_j}}}} +\mu_j( {\tau})  {{  z_j^\tau}}  ,\qquad {\rm where}\quad \frac {\partial}{\partial {\overline{{{z^\tau_j}}}} }:= \frac 12\left(\frac {\partial}{\partial y_{2j-1}^\tau}+i\frac {\partial}{\partial y_{2j }^\tau }\right).
\end{equation}

We will show that partial symbols
  of  complex vectors $Z_j^\tau$, $j=1,\ldots, n$,  act on Laguerre basis  simply  as shift operators in the following  Lemma.
As a corollary, the  partial   Fourier transformation of the sub-Laplacian is diagonal as shown in Subsection 6.
\begin{lem} \label{lem:Z-bar-partial-L}
   \begin{equation*}\label{eq:Z-L}\begin{split}
     \widetilde{\overline{Z_j^\tau}}  \widetilde{ \mathscr L}_{\mathbf{k} }^{ (-\mathbf{p}) }( y,\tau) &=\left\{\begin{array}{ll}-\sqrt {2\mu_j( {\tau})(k_j+p_j)}  \widetilde{ \mathscr L}_{\mathbf{k} }^{(-\mathbf{ p }+\mathbf{e}_j )}( y,\tau),\qquad \qquad& p_j=1,2,\ldots,\\ -\sqrt {2\mu_j( {\tau}) k_j }  \widetilde{ \mathscr L}_{\mathbf{k}-\mathbf{e}_j  }^{(-\mathbf{ p }+\mathbf{e}_j) }( y,\tau),\qquad& p_j=0,
    \end{array} \right. \\
      \widetilde{ {Z_j^\tau}}  \widetilde{ \mathscr L}_{\mathbf{k} }^{ (\mathbf{p}) }(y,\tau ) &=\left\{\begin{array}{ll} \sqrt { 2\mu_j( {\tau}) (k_j+1)  }   \widetilde{ \mathscr L}_{\mathbf{k}+ \mathbf{e}_j}^{(\mathbf{ p }-\mathbf{e}_j )}( y,\tau),\qquad\qquad \qquad& p_j=1,2,\ldots,\\ \sqrt { 2\mu_j( {\tau}) (k_j+1) }  \widetilde{ \mathscr L}_{\mathbf{k}  }^{(\mathbf{ p }-\mathbf{e}_j) }( y,\tau),\qquad\quad& p_j=0,
    \end{array} \right.
\end{split}\end{equation*}
where $\mathbf{e}_j=(0,\ldots,1,\ldots,0)$ with $1$ appearing in $j$-th entry and $0$ otherwise.
\end{lem}
\begin{proof} Recall that by definition (\ref{eq:exponential-Laguerre0})-(\ref{eq:exponential-Laguerre}), for $p,k=0,1,\ldots$,
\begin{equation*}
  \mu_j(\dot{\tau}) \widetilde{\mathscr  L}_{k }^{(-p)} \left( \sqrt {\mu_j(\dot{\tau})}{y^\tau_j} ,\tau\right)=\frac {2 \mu_j( {\tau})}{\pi }(-1)^p\left [\frac {\Gamma(k+1)}{\Gamma(k+p+1)}\right]^{\frac 12}
   L_k^{(p)} (\sigma) e^{-\frac \sigma2  }[2\mu_j( {\tau})]^{\frac p2}   (\overline{{{z^\tau_j}}})^p,
\end{equation*}
with complex $\tau$-coordinate  ${z^\tau_j} $   and
\begin{equation*}\label{eq:sigma}
   \sigma=2\mu_j( {\tau})  |{{z^\tau_j}}  |^2  .
\end{equation*}
Note that
\begin{equation*}
   \frac {\partial}{\partial {\overline{{{z^\tau_j}}}} }\sigma=2\mu_j( {\tau})  {{{z^\tau_j}} },\qquad \frac {\partial}{\partial {\overline{{{z^\tau_j}}}} }(\overline{{{z^\tau_j}}})^p=p(\overline{{{z^\tau_j}}})^{p-1},\qquad \overline{{z_j^\tau}}=|{{z^\tau_j}}  |e^{-i\theta} .
\end{equation*}
Then applying $\widetilde{Z_l^\tau}$ in (\ref{eq:Z-partial}), we get
\begin{equation*}
  \widetilde{\overline{Z_l^\tau}} L_{k }^{(p)} \left( \sigma\right)   = \left\{ 2 L_k^{(p) }{}'\left(\sigma\right) + L_k^{(p)}\left(\sigma \right)\right\}  \cdot \mu_j( {\tau})  {{{z^\tau_j}} }
\end{equation*}
and so
\begin{equation*}\begin{split}&
 \widetilde{\overline{Z_l^\tau}}\left [ \mu_j(\dot{\tau})\widetilde{\mathscr  L}_{k }^{(-p)} \left( \sqrt {\mu_j(\dot{\tau})}{\mathbf{y}^\tau_j} ,\tau \right)\right]   \\=&\delta_l^{(j)}  \left [\frac {\Gamma(k+1)}{\Gamma(k+p+1)}\right]^{\frac 12}\frac {2 \mu_j( {\tau})}{\pi }(-1)^p \left\{L_k^{(p) }{}'\left(\sigma\right) \sigma +p L_k^{(p)}\left(\sigma \right)\right\} e^{-\frac \sigma 2} [2\mu_j( {\tau})]^{\frac p2}   (\overline{{{z^\tau_j}}})^{p-1}\\
=& \left\{\begin{array}{ll}-\delta_l^{(j)}\sqrt {2\mu_j( {\tau})(k+p )}\cdot   \mu_j(\dot{\tau}) \widetilde{\mathscr  L}_{k }^{(-p+1)} \left( \sqrt {\mu_j(\dot{\tau})}{\mathbf{y}^\tau_j} ,\tau\right),\qquad \qquad& p =1,2,\ldots,\\   -\delta_l^{(j)}\sqrt { 2\mu_j( {\tau}) k  }  \cdot \mu_j(\dot{\tau})\widetilde{ \mathscr L}_{ {k}-1  }^{( 1) }\left( \sqrt {\mu_j(\dot{\tau})}{\mathbf{y}^\tau_j },\tau \right),& p =0,
    \end{array} \right.
\end{split}\end{equation*} by using the following identities (cf.   \cite[page 28]{BCT}) for Laguerre polynomials:
\begin{equation}\label{eq:dierivative-L}
   L_k^{(p) }{}'(\sigma)=\frac {d L_k^{(p) }}{d\sigma}(\sigma)=-L_{k-1}^{( p+ 1)}(\sigma)
\end{equation}
for $ p =0,1,\ldots,$ and
\begin{equation*}
 -\sigma L_{k-1}^{(p+1)}(\sigma)+p  L_k^{(p)}(\sigma)=(k+p)L_{k }^{(p-1)}(\sigma)
\end{equation*}
for $ p =1,2,\ldots.$
  For $ p =0$, we only need to use (\ref{eq:dierivative-L}).
So the first identity of (\ref{eq:Z-L}) is proved.

Similarly, we have
\begin{equation*}\label{eq:tilde-L}
  \mu_j(\dot{\tau}) \widetilde{\mathscr  L}_{k }^{(p)} \left( \sqrt {\mu_j(\dot{\tau})}{\mathbf y^\tau_j} ,\tau\right)=\frac {2 \mu_j( {\tau})}{\pi }\left [\frac {\Gamma(k+1)}{\Gamma(k+p+1)}\right]^{\frac 12}
   L_k^{(p)} (\sigma) e^{-\frac \sigma2  }[2\mu_j( {\tau})]^{\frac p2}   ({{z^\tau_j}})^p
\end{equation*}
and \begin{equation*}
   \frac {\partial}{\partial { {{{z^\tau_j}}}} }\sigma=2\mu_j( {\tau}) \overline{  z^\tau_j },\qquad \frac {\partial}{\partial { {{{z^\tau_j}}}} }({{z^\tau_j}})^p=p({{z^\tau_j}})^{p-1},\qquad \mu_j( {\tau}) \overline{{z}_j^\tau}({{z^\tau_j}})^p=\frac \sigma 2 ({{z^\tau_j}})^{p-1}.
\end{equation*}Then  applying $\widetilde{\overline{Z_l^\tau}}$ in (\ref{eq:Z-bar-partial}), we get
\begin{equation*}
 \widetilde{ {Z_l^\tau}} L_{k }^{(p)} \left( \sigma\right)   = \left\{2  L_k^{(p) }{}'\left(\sigma\right)  - L_k^{(p)}\left(\sigma \right)\right\}  \cdot \mu_j( {\tau}) \overline{  z^\tau_j }
\end{equation*}
and so
\begin{equation*}\begin{split}&
  \widetilde{ {Z_l^\tau}}\left [ \mu_j(\dot{\tau})\widetilde{\mathscr  L}_{k }^{( p)} \left(\sqrt {\mu_j(\dot{\tau})}{\mathbf{y}^\tau_j}, \tau \right)\right] \\ =& \delta_l^{(j)} \left [\frac {\Gamma(k+1)}{\Gamma(k+p+1)}\right]^{\frac 12} \frac {2 \mu_j( {\tau})}{\pi}\left\{\left[ L_k^{(p) }{}'\left(\sigma\right) -L_k^{(p) }\left(\sigma\right) \right]\sigma+pL_k^{(p)} (\sigma  )\right \}  e^{-\frac \sigma 2  } [2\mu_j( {\tau})]^{\frac  {p }2}   ( {{{z^\tau_j}}})^{p-1} \\
=& \left\{\begin{array}{ll} \delta_l^{(j)}\sqrt { 2\mu_j( {\tau}) (k+1 ) } \cdot \mu_j(\dot{\tau}) \widetilde{\mathscr  L}_{k +1}^{( p-1)} \left( \sqrt {\mu_j(\dot{\tau})}{\mathbf{y}^\tau_j },\tau \right),\qquad \qquad& p =1,2,\ldots,\\ \delta_l^{(j)}\sqrt { 2\mu_j( {\tau}) (k +1) }\cdot \mu_j(\dot{\tau})  \widetilde{ \mathscr L}_{ {k}  }^{  ( -1)}\left( \sqrt {\mu_j(\dot{\tau})}{\mathbf{y}^\tau_j } ,\tau\right),& p =0,
    \end{array} \right.
\end{split}\end{equation*} by using (\ref{eq:dierivative-L}) and identities (cf. \cite{BCT})
\begin{equation*}\begin{split}
 -\sigma L_{k-1}^{(p+1)}(\sigma)-\sigma L_{k }^{(p )}(\sigma)+p  L_k^{(p)}(\sigma)& =(k+1 )L_{k+1 }^{(p-1)}(\sigma),
\end{split}\end{equation*}for $ p =1,2,\ldots,$   which follows from taking derivatives in the both sides of  the generating function formula
(\ref{eq:generating-function}) with respect to $z$, and $
   L_k^{(p) }(\sigma)+L_{k -1}^{( p+ 1)}(\sigma)=L_{k }^{( p+ 1)}(\sigma)
$ for $p=0$.
\end{proof}

\section{Applications}
\subsection{The fundamental solution to the sub-Laplacian }
Define the sub-Laplacian
\begin{equation*}
\triangle_b:=-\frac 14\sum_{k=1}^{{2n}} Y_{k} Y_{ k}.
\end{equation*}

\begin{prop} \label{prop:ZvZv=ZeZe} For any given $\tau\in  \mathbb{R}^r \setminus\{0\}$ with $B^\tau $  non-degenerate, let $\{v^\tau_1,\ldots, v^\tau_{{2n}}\}$ be the local orthonormal basis  of
$ \mathbb{R}^{{2n}}$ as before. Then, we have
   \begin{equation*} \triangle_b =-\frac 14 \sum_{k=1}^{{2n}} Y_{v^\tau_k}
Y_{v^\tau_k}.  \label{eq:ZvZv=ZeZe}
\end{equation*}
\end{prop}
\begin{proof} For  given $\tau\in  \mathbb{R}^r \setminus\{0\}$, we write
$
   v^\tau_j=\left(a_{j1},\ldots, a_{j(2n)}\right).
$ Then $(a_{jk})$ is an orthogonal  matrix since $\{v^\tau_1,\ldots, v^\tau_{{2n}}\}$ is an orthonormal basis  of
$ \mathbb{R}^{{2n}}$, and so we have
$
   \sum_{k=1}^{2n} a_{kl}a_{km}=\delta^{(l)}_{ m}.
$
It follows from  (\ref{eq:Y})
 that
 \begin{equation*}\sum_{k=1}^{{2n}} Y_{v^\tau_k}
Y_{v^\tau_k}=\sum_{k=1}^{{2n}} \left(\sum_{l=1}^{{2n}}a_{kl}Y_{l}\right)^2
 =\sum_{l,m=1}^{ 2n}\sum_{k=1}^{{2n}} a_{kl} a_{km}Y_{ l}
Y_{ m}=\sum_{k=1}^{{2n}} Y_{k}
Y_{ k}.
\end{equation*}
The result is proved.
 \end{proof}

It follows from Proposition \ref{prop:ZvZv=ZeZe} that   for any fixed $\tau\in  \mathbb{R}^r \setminus\{0\}$, we have
\begin{equation*}
   \triangle_b=-\frac 14 \sum_{j=1}^{{2n}} Y_{v^\tau_{j}}
Y_{v^\tau_{j}}= -\frac 12\sum_{j=1}^{{ n}} ( {Z_j^\tau }
      {\overline{Z_j^\tau }} +
       {\overline{Z_j^\tau }} {Z_j^\tau }),
\end{equation*}
and its partial symbol is
\begin{equation*}\label{eq:triangle-tilde}\widetilde{\triangle_b}=-\frac 14\sum_{j=1}^{{2n}} \widetilde{Y_{ j}}
\widetilde{Y_{ j}}= -\frac 14
  \sum_{j=1}^{{2n}}  \widetilde{ Y_{v^\tau_j}}
 \widetilde{ Y_{v^\tau_j}}= - \frac 12\sum_{j=1}^{{ n}} \left(  \widetilde{Z_j^\tau }
      \widetilde {\overline{Z_j^\tau }} +
       \widetilde {\overline{Z_j^\tau }}  \widetilde {Z_j^\tau }\right).
\end{equation*}
By formula (\ref{eq:Z-L}), we find that
\begin{equation}\label{eq:subL-Laguerre}
-\frac 12  \left ( \widetilde{Z_j^\tau}
      \widetilde{\overline{Z_j^\tau}} +
      \widetilde{\overline{Z_j^\tau}} \widetilde{Z_j^\tau} \right  ) \widetilde{ \mathscr L}_{\mathbf{k} }^{ ( \mathbf{0}) }(y,\tau )=
     \mu_j( {\tau})  (2k_j+1)\widetilde{ \mathscr L}_{\mathbf{k} }^{ (\mathbf{0}) }(y,\tau ),
\end{equation}
 Thus its   Laguerre tensor is
\begin{equation*}\label{eq:Laguerre-subL}
  \mathcal{M}_\tau\left( {\triangle_b}\right)=\left(\sum_{j=1}^n   \mu_j( {\tau})  (2k_j+1)\delta_{k_1}^{(p_1)}\cdots\delta_{k_n}^{(p_n)} \right).
\end{equation*}
Its inverse Laguerre tensor is
\begin{equation*}\label{eq:inverse-Laguerre-subL}
  \mathcal{M}_\tau\left( {\triangle_b^{-1}}\right)=\left(\left[\sum_{j=1}^n   \mu_j( {\tau})  (2k_j+1)\right]^{-1}\delta_{k_1}^{(p_1)}\cdots\delta_{k_n}^{(p_n)}\right).
\end{equation*}

 Since   $\widetilde{ \mathscr L}_{\mathbf{k} }^{(\mathbf{p} )}( y,\tau)$ is a locally integrable  function over $\mathbb{R}^{{2n+r}}$ by Lemma \ref{lem:L2}, it follows from the  continuity of  the inverse
partial Fourier transformation  that  $ { \mathscr L}_{\mathbf{k} }^{(\mathbf{p} )}( y,\tau)$ is a tempered distribution on $\mathbb{R}^{{2n+r}}$. Now
consider a tempered distribution $ { \mathscr F}_{\mathbf{k} }  $,  whose partial Fourier transformation is
\begin{equation*}
 \widetilde{ \mathscr F}_{\mathbf{k} } ( y,\tau)=    \frac 1{\sum_{j=1}^n  \mu_j(  {{\tau}})  (2k_j+1) }  \widetilde{ { \mathscr L}}_{\mathbf{k} }^{(\mathbf{0} )}( y,\tau).
\end{equation*}Note that
\begin{equation*}
  \left \|\widetilde{ \mathscr F}_{\mathbf{k} } ( \cdot,\tau)\right\|_{L^2(\mathbb{R}^{ 2n })}=\frac { {\prod_{j=1}^n \mu_j( {\tau})}}{\sum_{j=1}^n  \mu_j(  {{\tau}})  (2k_j+1) } \left(\frac {2  }{\pi}\right)^{\frac n2}\leq \frac { {\prod_{j=1}^n \mu_j( {\tau})}}{\sum_{j=1}^n  \mu_j(  {{\tau}})    } \left(\frac {2  }{\pi}\right)^{\frac n2},
\end{equation*} by Lemma \ref{lem:L2}.  Since $B^\tau $ is non-degenerate for   all $0\neq\tau\in \mathbb{R}^r$, the upper bound
of the above $L^2$ norm is $C|\tau|^{ n-1}$ for some $C>0$  independent of $\mathbf{k}$ and $\tau$. Therefore,
  $\widetilde{ \mathscr F}_{\mathbf{k} } (y,\tau )$ is locally integrable.
Define
\begin{equation}\label{eq:kernel}
 {\Psi}_R: = \sum_{|\mathbf{k}|=0}^\infty  { \mathscr F}_{\mathbf{k} }   R^{\mathbf{k}},\qquad R \in (-1,1).
\end{equation}  It is a tempered distribution because its   partial Fourier transformation
$
   \sum_{|\mathbf{k}|=0}^\infty  \widetilde{{ \mathscr F}}_{\mathbf{k} } ( \cdot,\cdot) R^{\mathbf{k}}
$ is locally integrable by using the above argument.

 Now we have
$
     {\triangle_b}  { \mathscr F}_{\mathbf{k}} ={ \mathscr L}_{\mathbf{k} }^{ (\mathbf{0}) },
$
due to
\begin{equation*}
  \widetilde{ {\triangle_b}  { \mathscr F}_{\mathbf{k} }} = \frac 1{\sum_{j=1}^n  \mu_j(  {{\tau}})  (2k_j+1) }
\sum_j\frac 12 \left ( \widetilde{Z_j^\tau}
      \widetilde{\overline{Z_j^\tau}} +
      \widetilde{\overline{Z_j^\tau}} \widetilde{Z_j^\tau} \right  ) \widetilde{ \mathscr L}_{\mathbf{k} }^{ ( \mathbf{0}) }( y,\tau)=
  \widetilde{ \mathscr L}_{\mathbf{k} }^{ (\mathbf{0}) }( y,\tau)
\end{equation*}
implied by (\ref{eq:subL-Laguerre}).
Then it follows from   Theorem \ref{thm:app} that
\begin{equation}\label{eq:delta}
   {\triangle_b} {\Psi}_R (y,s )=\sum_{|\mathbf{k}|=0}^\infty  {\triangle_b}  { \mathscr F}_{\mathbf{k} }  R^{\mathbf{k}}=\sum_{|\mathbf{k}|=0}^\infty   { \mathscr L}_{\mathbf{k} }^{ (\mathbf{0}) }( y,s)R^{\mathbf{k}}\rightarrow \delta,\qquad \mbox{ as   } R\rightarrow 1-,
\end{equation}
because ${\triangle_b} $
continuous on the space $\mathcal{ S}'(\mathbb{R}^{{2n+r}})$ of tempered distributions, since
differentiation and multiplication by  a polynomial are continuous  on the space $\mathcal{ S}'(\mathbb{R}^{{2n+r}})$. We claim that
\begin{equation}\label{eq:claim}
 \Psi:= \lim_{R\rightarrow 1-} {\Psi}_R \,\, {\rm is }\,\,{\rm  a } \,\,{\rm tempered } \,\, {\rm distribution}.
\end{equation}
Then we get
\begin{equation*}
    {\triangle_b} {\Psi} ={\triangle_b}\lim_{R\rightarrow 1-} {\Psi}_R  =\lim_{R\rightarrow 1-} {\triangle_b} {\Psi}_R  =\delta
\end{equation*} by the continuity of $ {\triangle_b}$ on   $\mathcal{ S}'(\mathbb{R}^{{2n+r}})$. Thus  ${\Psi}$ is the fundamental solution to the sub-Laplacian.

Let us prove the claim (\ref{eq:claim}) and calculate $\Psi$. By definition (\ref{eq:kernel}), we have
\begin{equation*}\label{eq:kernel'}
  \widetilde{\Psi}_R (y,\tau)=\frac 1{|\tau|}\sum_{|\mathbf{k}|=0}^\infty \frac 1{\sum_{j=1}^n  \mu_j( \dot{{\tau}})  (2k_j+1) }\prod_{j=1}^n\mu_j( \dot{{\tau}})   \widetilde{\mathscr  L}_{k_j}^{(0)} \left( \sqrt {\mu_j(\dot{\tau})}{\mathbf{y}^\tau_j},\tau \right)R^{k_j}
\end{equation*} Note that $\frac 1A=\int_0^\infty e^{-As}ds$ for  $A>0$.
Then  by $\mu_j(\dot{\tau})\neq 0$, we get
\begin{equation*}\begin{split}
   \widetilde{\Psi}_R (y,\tau)&=\frac 1{|\tau|}\sum_{|\mathbf{k}|=0}^\infty\int_0^\infty e^{- \left[\sum_{j=1}^n   \mu_j( \dot{{\tau}})  (2k_j+1)\right]s} \prod_{j=1}^n\mu_j( \dot{{\tau}})   \widetilde{\mathscr  L}_{k_j}^{(0)} \left( \sqrt {\mu_j(\dot{\tau})}{\mathbf{y}^\tau_j},\tau \right)R^{k_j}ds \\
   &=\frac{ {|\tau|}^{n-1}}{\pi^n}\sum_{|\mathbf{k}|=0}^\infty\int_0^\infty e^{- \left[\sum_{j=1}^n    \mu_j( \dot{{\tau}})  (2k_j+1)\right]s} \prod_{j=1}^n 2 \mu_j( \dot{{\tau}})    {  l}_{k_j}^{(0)} \left(2  {\mu_j( {\tau})}|{z^\tau_j}|^2 \right)R^{k_j}ds \\&
   =\frac{ {|\tau|}^{n-1}}{\pi^n} \int_0^\infty \prod_{j=1}^n 2\mu_j( \dot{{\tau}})  e^{-      \mu_j( \dot{{\tau}})   s }\sum_{  {k}_j =0}^\infty\left(e^{-   2  \mu_j( \dot{{\tau}})   s}R\right)^{k_j}  {  l}_{k_j}^{(0)} \left(2 {\mu_j( {\tau})}|{z^\tau_j}|^2 \right)ds\\&
   =\frac{ {|\tau|}^{n-1}}{\pi^n} \int_0^\infty \prod_{j=1}^n \frac {2\mu_j( \dot{{\tau}})e^{-      \mu_j( \dot{{\tau}})   s} }{1-e^{-   2  \mu_j( \dot{{\tau}})   s}R} \exp\left\{ -   {\mu_j( {\tau})}|{z^\tau_j}|^2\cdot \frac {1+  e^{-   2   \mu_j( \dot{{\tau}})   s}R}{1-e^{-   2  \mu_j( \dot{{\tau}})   s}R} \right\} ds
    \end{split}
\end{equation*}by the  generating function formula
(\ref{eq:generating-function'}) for $l_{k}^{(0)}$.
Take partial inverse Fourier transformation and use Fubini's Theorem to get
\begin{equation*}\begin{split}
& {\Psi}_R (y,t)= \int_{\mathbb{R}^r} e^{-it\cdot\tau}\widetilde{\Psi}_R(y,\tau)d\tau
   \\& = \int_0^\infty \frac{|\tau|^{n+r-2}}{\pi^n}d|\tau| \int_0^\infty\frac { ds}{s^n} \int_{S^{r-1}}d\dot{\tau}  e^{-i|\tau| t\cdot\dot{\tau}}\prod_{j=1}^n \frac {2 \mu_j(  {\dot{{\tau}}}) s }{e^{   \mu_j( \dot{{\tau}})   s}-e^{-     \mu_j( \dot{{\tau}})   s}R} e^{  -  |\tau| {\mu_j( \dot{{\tau}})}|{z^\tau_j}|^2\cdot\frac {1+  e^{-   2   \mu_j( \dot{{\tau}})   s}R}{1-e^{-   2  \mu_j( \dot{{\tau}})   s}R}  }.
    \end{split}
\end{equation*}
Set
\begin{equation*}
  \lambda= \dot{\tau}s\in\mathbb{ R}^r.
\end{equation*}  Then $\dot{\lambda}=\dot{\tau}$, $ B^\lambda  =s    B^{\dot{\tau}}  $ and ${z_j^\lambda}={z^\tau_j}$. Since  $v^\tau_j$  is the eigenvector of the symmetric matrix $|B^\lambda|^2=  (B^\lambda)^t  B^\lambda $, and so it  is the eigenvector of $|B^{\dot{\lambda}}|  \frac {1+  e^{-   2|B^\lambda|}R}{1- e^{-   2|B^\lambda|}R}  $. If we write $y=\sum_j y^\tau_j v^\tau_j$, we get
\begin{equation*}
\sum_{j=1}^n  {\mu_j( \dot{{\tau}})}|{z^\tau_j}|^2\cdot\frac {1+  e^{-   2   \mu_j( \dot{{\tau}})   s}R}{1-e^{-   2  \mu_j( \dot{{\tau}})   s}R} =\left\langle \left|B^{\dot{\lambda}}\right| \cdot\frac {1+  e^{-   2|B^\lambda|}R}{1- e^{-   2|B^\lambda|}R} y,y\right\rangle=:B(y,\lambda;R).
\end{equation*}Write $C (\lambda;R) : =\prod_{j=1}^n (e^{   \mu_j( \lambda)    }-e^{-      \mu_j( \lambda)   }R)$.
Then we integrate out the variable $|\tau|$ to get
\begin{equation}\label{eq:psi-R}
\begin{split}
{\Psi}_R(y,t) &
    =\frac  {2^n}{\pi^n} \int_0^\infty \frac {d|\lambda|}{|\lambda|^n} \int_{S^{r-1}}d\dot{\lambda} \frac { \det  |B^\lambda|^{\frac 12}     } {C (\lambda;R) }\int_0^\infty |\tau|^{n+r-2} e^{-  |\tau| \left[i t\cdot \dot{\lambda}  +  B(y,\lambda;R) \right] }d|\tau|\\ &
    =\frac  {{2^n}\Gamma(n+r-1)}{\pi^n} \int_0^\infty \frac {|\lambda|^{r-1}d|\lambda|}{|\lambda|^{n+r-1}} \int_{S^{r-1}}d\dot{\lambda}  \frac { \det  |B^\lambda|^{\frac 12}     } {C (\lambda;R) }\frac 1 { ( B(y,\lambda;R)+it \cdot \dot{{\lambda}})^{n+r-1} }\\
     &=\frac{{2^n}\Gamma(n+r-1)}{\pi^n}   \int_{\mathbb{R}^r}   \frac { \det  |B^\lambda|^{\frac 12}     } {C (\lambda;R) }\frac { d\lambda} { ( |\lambda| B(y,\lambda;R)+it\cdot {\lambda})^{n+r-1}  }
   \end{split}
\end{equation}
by using Fubini's theorem, and
$
 \Gamma(m)  A^{-m}=\int_0^\infty s^{m-1} e^{-As}ds
$ for ${\rm Re }A>0$.
Noting that
\begin{equation*}\begin{split}&
 \lim_{R\rightarrow 1-} |\lambda| B(y,\lambda;R)=  \langle |B^\lambda| \coth|B^\lambda| y,y\rangle,\\
 & \lim_{R\rightarrow 1-} C (\lambda;R)  = \det  (2\sinh    |B^\lambda|^{\frac 12})=2^n\det  (\sinh    |B^\lambda|^{\frac 12}) ,
\end{split}\end{equation*}we see that the last integral in (\ref{eq:psi-R}) is absolutely convergent for $y\neq 0$ when $\mu_j(\dot{\tau}) $'s have a positive lower bound. Then
we get\begin{equation*}\begin{split} {\Psi} (y,t) &
     =
\lim_{R\rightarrow 1-}{\Psi}_R(y,t)\\&
     =\frac{\Gamma(n+r-1)}{\pi^n}   \int_{\mathbb{R}^r}    \det\left[\frac  { |B^\lambda| }{ \sinh   |B^\lambda|    } \right]^{\frac 12}\frac { d\lambda} { ( \langle |B^\lambda| \coth|B^\lambda| y,y\rangle+it\cdot {\lambda})^{n+r-1}  }.
    \end{split}
\end{equation*}
For $y=0$, it is standard to use analytic continuation (cf. e.g. \cite{ww}).  We omit details.

   \subsection{The Szeg\"o kernel for $k$-CF  functions on the quaternionic Heisenberg group}
   The $7$-dimensional  quaternionic Heisenberg group $\mathscr{H}$ is the nilpotent group $\mathbb{R}^4\times \mathbb{R}^3$ with $B$ given by (\ref{eq:quaternionic-Heisenberg}).
Recall    the \emph{tangential $k$-Cauchy-Fueter operator} \cite{SW}
\begin{equation*}\begin{aligned}\label{cf}
\mathscr{D}^{(k)}_{b}:C^{\infty}(\mathscr{H}, \mathbb{C}^{k+1})
&\rightarrow C^{\infty}(\mathscr{H}, \mathbb{C}^{2k}
 )\qquad k=1,2,\cdots.
\end{aligned}\end{equation*}
   A $\mathbb{C}^{k+1}$-valued distribution $f$ on $\mathscr{H}$ is called \emph{$k$-CF}  if $\mathscr{D}^{(k)}_{b}f=0$ in the sense of distributions. The tangential $k$-Cauchy-Fueter operator and the $k$-CF functions on
the quaternionic Heisenberg group are quaternionic counterparts of the
tangential CR operator and CR functions on the Heisenberg group in the
theory of several complex valuables. Consider the space of  $L^{2}$-integrable  $k$-CF  functions
\begin{align*}
\mathcal{A}(\mathscr{H},\mathbb{C}^{k+1})=\left\{f\in L^{2}(\mathscr{H},\mathbb{C}^{k+1});\mathscr{D}^{(k)}_{b}f=0\right\}.
\end{align*} The  orthogonal  projection operator
\begin{align*}\label{1.11}
P:L^{2}(\mathscr{H},\mathbb{C}^{k+1})\longrightarrow
\mathcal{A}(\mathscr{H},\mathbb{C}^{k+1})
\end{align*}is called the \emph{Szeg\"o projection operator}. We will drop superscripts $(k)$ for simplicity. The Szeg\"o kernel of $P$ is given by the following theorem.

\begin{thm}\label{main} {\rm (Theorem 1.1 in \cite{SW})}
The Szeg\"o kernel of the Szeg\"o projection is an  End$(\mathbb{C}^{k+1})$-valued homogeneous function
 \begin{equation*}
    S(y,s)=
\frac{2^7\cdot 3}{(2\pi)^{5}}\int_{S^{2}}\frac{P^{{\tau}}}
{\left(|y|^{2}-i\tau\cdot s\right)^5}
\hbox{d}\tau,
 \end{equation*}
for $y\neq0,$ where $P^{ \tau} $ is the orthogonal projection to vector $e_{1}^{\tau}$ given by (\ref{e1}). In general, for $(y,s)\neq(0,0),$ it is given by the integral with changed contour.
\end{thm}

In the main step of the proof of this theorem, the group Fourier transformation is used. This part can be simplified by using  the Laguerre calculus as follows.
To find the kernel of $P,$  we consider the associated operator of the second order  $\Box_{b}:=\mathscr{D}_{b}^{*}
\mathscr{D}_{b},$  where $\mathscr{D}_{b}^{*}$ is the adjoint operator of $\mathscr{D}_{b}.$    The explicit expression of the associated operator $\Box_b $ is known~\cite{SW} as
\begin{equation*}
   \Box_{b}
= 4\Delta_{b} \cdot M_{k}+8D_k,
\end{equation*}
where the  sub-Laplacian $\Delta_{b} =- \frac 14 \sum_{j=1}^{4}Y_{j}^{2}$ is different from that in \cite{SW} with a factor $\frac 14$, and \begin{align*}\label{ek}
M_{k}:=\left(\begin{smallmatrix} 1& &&&\\ & 2 &&&\\&&\ddots&&\\ & && 2 &\\ &&&&1\end{smallmatrix}\right)
\end{align*} and
 \begin{equation*}
 D_k=\left(\begin{smallmatrix} i\partial_{s_{1}}& \partial_{s_{2}}-i\partial_{s_{3}}&0&\cdots&0&0&0\\-\partial_{s_{2}}
-i\partial_{s_{3}}& 0 &\partial_{s_{2}}
-i\partial_{s_{3}}&\cdots&0&0&0\\0&-\partial_{s_{2}}
-i\partial_{s_{3}}& 0 &\cdots&0&0&0\\ \vdots&\vdots&\vdots&\ddots&\vdots&\vdots&\vdots
\\0&0&0&\cdots& 0 &\partial_{s_{2}}
-i\partial_{s_{3}}&0 \\0&0&0&\cdots&-\partial_{s_{2}}
-i\partial_{s_{3}}&0&\partial_{s_{2}}-i\partial_{s_{3}}\\
0&0&0&\cdots&0&-\partial_{s_{2}}
-i\partial_{s_{3}}&-i\partial_{s_{1}}\end{smallmatrix}\right)
 \end{equation*}
  are $(k+1)\times(k+1)$ matrices. We know that
\begin{align*}\label{5.3}\mathcal{A}(\mathscr{H},\mathbb{C}^{k+1})=
{\rm ker}\, \Box_{b}:=\left\{f\in L^{2}(\mathscr{H},\mathbb{C}^{k+1});\mathscr{D}_{b}f\in L^{2}(\mathscr{H},\mathbb{C}^{k+1}),\mathscr{D}_{b}^*
\mathscr{D}_{b}f=0\right\}
\end{align*}by Corollary 4.1 in \cite{SW},
where $\mathscr{D}_{b}^*\mathscr{D}_{b}f=0$ holds in the sense of distributions.
  \begin{prop}\label{prop:eigen} {\rm (Lemma 3.1 in \cite{SW})} The matrix $|\tau|M_k+D_k^{{\tau}}$
for any ${\tau}\in \mathbb{R}^3\setminus \{0\}$ is semi-positive with only one eigenvalue vanishing, whose eigenvector is
\begin{equation}\label{e1}
e_{1}^{\tau}:=\frac{1}{\gamma}\cdot\left(\begin{array}{l}
(1+\tau_{1})^{k}\\
({1+\tau_{1}})^{k-1}({i\tau_{2}-\tau_{3}})\\\ \ \ \ \ \ \   \vdots\\
(1+\tau_{1})({i\tau_{2}-\tau_{3}})^{k-1}\\
({i\tau_{2}-\tau_{3}})^{k}
\end{array}\right)\quad {\rm with}\quad \gamma=\left(\sum_{j=0}^k(1+\tau_1)^{2k-j}
(1-\tau_1)^{j}\right)^{\frac{1}{2}}
\end{equation}
if $\tau\neq(-1,0,0);$ and $e_{1}^{\tau}:=(0,\cdots,0,1)^t,$ if $\tau=(-1,0,0).$
\end{prop}

We identify $L^2(\mathbb{R}^{4 },\mathbb{C}^{k+1})$ with $L^2(\mathbb{R}^{4 }) \otimes\mathbb{C}^{k+1} $. For fixed ${\tau}\in \mathbb{R}^3\setminus \{0\}$, we write $ \widetilde{ \mathscr L}_{\mathbf{p}-1 }^{(\mathbf{0} )}(\cdot,\tau)$ as $\widetilde{ \mathscr L}_{\mathbf{p}-1 }^{(\mathbf{0} )}$ in the sequel. Define
\begin{equation*}
   V_{\mathbf{p}}^\tau:= L^2(\mathbb{R}^{4 })*_\tau\widetilde{ \mathscr L}_{\mathbf{p}-1  }^{(\mathbf{0} )}  .
\end{equation*}

Recall that $\mu_j(   \tau )=|\tau|$ for  the quaternionic Heisenberg group. Since $\widetilde{ \mathscr L}_{\mathbf{p}-1 }^{(\mathbf{0} )}(\cdot,\tau )$ is a
rapidly decreasing smooth function for fixed $\tau\neq 0$ by definition,
\begin{equation*}
   \Pi_\tau:= *_\tau \widetilde{ \mathscr L}_{\mathbf{p}-1 }^{(\mathbf{0} )} : L^2(\mathbb{R}^{4 },\mathbb{C}^{k+1})\rightarrow L^2(\mathbb{R}^{4 },\mathbb{C}^{k+1})
\end{equation*}
is a bounded operator, Moreover, it is a projection, i.e.
$\Pi_\tau^2=\Pi_\tau$,  because of
\begin{equation*}
\left(  f*_\tau  \widetilde{ \mathscr L}_{\mathbf{p}-1 }^{(\mathbf{0} )}\right)*_\tau\widetilde{ \mathscr L}_{\mathbf{p}-1 }^{(\mathbf{0} )} =f*_\tau\left(\widetilde{ \mathscr L}_{\mathbf{p}-1 }^{(\mathbf{0} )}*_\tau \widetilde{ \mathscr L}_{\mathbf{p}-1 }^{(\mathbf{0} )}\right) = f*_\tau \widetilde{ \mathscr L}_{\mathbf{p}-1 }^{(\mathbf{0} )}.
\end{equation*} Here we used Proposition \ref{prop:vector-tau} (3).
Note that
\begin{equation}\label{eq:L-k-1-0}\begin{split}
   &   \widetilde{ \mathscr L}_{\mathbf{k}\wedge \mathbf{p}'-\mathbf{1} }^{(\mathbf{k }-\mathbf{ p}' )}*_\tau\widetilde{ \mathscr L}_{\mathbf{p}-1 }^{(\mathbf{0} )}=\delta_{\mathbf{ p}'}^{(\mathbf{ p} )} \widetilde{ \mathscr L}_{\mathbf{k}\wedge \mathbf{p}-\mathbf{1} }^{(\mathbf{k }-\mathbf{ p} )},
\end{split}\end{equation}
by  Proposition \ref{prop:product}. Therefore
\begin{equation*}
   V_{\mathbf{p}}^\tau =\mbox{   span    } \left \{\widetilde{ \mathscr L}_{\mathbf{k}\wedge \mathbf{p}-\mathbf{1} }^{(\mathbf{k }-\mathbf{ p} )};\mathbf{ k}\in \mathbb{Z}_+^2 \right\},
\end{equation*} which is an infinite dimensional space.
  Then the following decomposition   follows from the fact that $ \{\widetilde{ \mathscr L}_{\mathbf{k}\wedge \mathbf{p}-\mathbf{1} }^{(\mathbf{k }-\mathbf{ p} )}  \}$ in (\ref{eq:exponential-Laguerre}) is a    basis of $L^2(\mathbb{R}^{4})$ for any fixed $\tau\in \mathbb{R}^r\setminus \{0\}$.
\begin{prop} For any ${\tau}\in \mathbb{R}^3\setminus \{0\}$, we have
   \begin{equation*}\label{eq:decomposition}
      L^2(\mathbb{R}^{4 },\mathbb{C}^{k+1}) \cong \bigoplus_{ |\mathbf{p}|=1 }^\infty V_{\mathbf{p}}^\tau\otimes \mathbb{C}^{k+1}.
   \end{equation*}
\end{prop}

 \begin{prop}\label{prop:ker} We have
\begin{equation}\label{eq:ker}
  \ker \widetilde{\Box}_{b}= V_{\mathbf{1}}^\tau\otimes e_{1}^{{\tau}},
\end{equation}where $V_{\mathbf{1}}^\tau$ is spanned by $   \{\widetilde{ \mathscr L}_{\mathbf{0} }^{(\mathbf{k }-\mathbf{ 1} )};\mathbf{ k}\in \mathbb{Z}_+^2  \}$.
\end{prop}
\begin{proof} By definition, the   partial symbol of $\Box_{b} $ is
$
  4M_{k}\widetilde{\triangle_b}+8iD_k^{\tau}
$
where
 \begin{align*}\label{dk}
D_{k}^{\tau}:=\left(\begin{smallmatrix} i\tau_{1}& \tau_{{2}}-i\tau_{{3}}&0&\cdots&0&0&0\\-\tau_{{2}}
-i\tau_{{3}}& 0 &\tau_{{2}}
-i\tau_{{3}}&\cdots&0&0&0\\0&-\tau_{{2}}
-i\tau_{{3}}& 0 &\cdots&0&0&0\\ \vdots&\vdots&\vdots&\ddots&\vdots&\vdots
&\vdots\\0&0&0&\cdots&0&\tau_{{2}}
-i\tau_{{3}}&0\\0&0&0&\cdots&-\tau_{{2}}
-i\tau_{{3}}& 0 &\tau_{{2}}
-i\tau_{{3}} \\0&0&0&\cdots&0&-\tau_{{2}}
-i\tau_{{3}}&-i\tau_{{1}}\end{smallmatrix}\right)
\end{align*}
is a $(k+1)\times(k+1)$ matrix for $\tau\in\mathbb{R}^3.$ Note that
\begin{equation*}\begin{split}
 \widetilde{\triangle}_{b} \widetilde{ \mathscr L}_{\mathbf{k}\wedge \mathbf{p}-\mathbf{1} }^{(\mathbf{k }-\mathbf{ p} )}   &=\widetilde{\triangle}_{b}\left (\widetilde{ \mathscr L}_{\mathbf{k}\wedge \mathbf{p}-\mathbf{1} }^{(\mathbf{k }-\mathbf{ p} )}*_\tau\widetilde{ \mathscr L}_{\mathbf{p}-1 }^{(\mathbf{0} )}\right) =\widetilde{ \mathscr L}_{\mathbf{k}\wedge \mathbf{p}-\mathbf{1} }^{(\mathbf{k }-\mathbf{ p} )} *_\tau\left(
 \widetilde{\triangle}_{b} \widetilde{ \mathscr L}_{\mathbf{p}-1 }^{(\mathbf{0} )}
 \right)\\&=2|\tau|  ( |\mathbf{ p} |-1) \widetilde{ \mathscr L}_{\mathbf{k}\wedge \mathbf{p}-\mathbf{1} }^{(\mathbf{k }-\mathbf{ p} )}*_\tau\widetilde{ \mathscr L}_{\mathbf{p}-1 }^{(\mathbf{0} )}=    2|\tau|  ( |\mathbf{ p} |-1) \widetilde{ \mathscr L}_{\mathbf{k}\wedge \mathbf{p}-\mathbf{1} }^{(\mathbf{k }-\mathbf{ p} )}
\end{split}\end{equation*}by using Proposition \ref{prop:vector-tau}, (\ref{eq:subL-Laguerre}) and (\ref{eq:L-k-1-0}), where $|\mathbf{ p} |=p_1+p_2$.
Thus for $v\in \mathbb{C}^{k+1}$,
\begin{equation*}\label{eq:coefficients}\begin{split}
   \left (4M_{k}\widetilde{\triangle_b}+8iD_k^{\tau}\right) \left(\widetilde{ \mathscr L}_{\mathbf{p}\wedge\mathbf{p}-1 }^{(\mathbf{p}-\mathbf{k} )}( y,\tau)\otimes v\right)  =  \widetilde{ \mathscr L}_{\mathbf{p}\wedge\mathbf{p}-1 }^{(\mathbf{p}-\mathbf{k} )}( y,\tau)\otimes 8\left ( |\tau|  ( |\mathbf{ p} |-1) M_{k} + iD_k^{\tau}\right)v  .
\end{split} \end{equation*}
The symmetric matrix in the right hand side is $ 8(|\tau| M_{k} + iD_k^{\tau})+8\sum_{j=1}^2     ( p_j-1) |\tau| M_{k} $. It follows from Proposition \ref{prop:eigen} that for $\mathbf{ p}\in \mathbb{Z}_+^2$ with $\mathbf{ p}\neq \mathbf{1}$, it has $k+1$ eigenvectors with positive eigenvalues, while for   $\mathbf{ p}= \mathbf{1}$, it has $k$
eigenvectors with positive eigenvalues and only one eigenvector $e_{1}^{{\tau}}$ with vanishing eigenvalue. In summary, we find a basis of $L^2(\mathbb{R}^{4 },\mathbb{C}^{k+1})$, consisting of smooth functions, such that $\widetilde{\Box}_{b}$ acts diagonally and (\ref{eq:ker}) holds.
\end{proof}

For any $g\in L^2(\mathscr{H},\mathbb{C}^{k+1})$, we have Laguerre expansions
\begin{equation}\label{eq:pg-tau}
    \widetilde{g}_\tau=\sum_{\mathbf{k} ,\mathbf{ p},j } G_{\mathbf{k}}^{\mathbf{p} }(\tau)\widetilde{ \mathscr L}_{\mathbf{k}\wedge \mathbf{p}-\mathbf{1} }^{(\mathbf{k }-\mathbf{ p} )} ,\qquad
{\rm  and }
 \qquad
  ( \widetilde{Pg})_\tau=\sum_{\mathbf{ q} } C_{ \mathbf{q}}^{\mathbf{1} } (\tau)\widetilde{ \mathscr L}_{\mathbf{0}  }^{(\mathbf{q }-\mathbf{ 1} )}\otimes e_{1}^{{\tau}},
\end{equation} by Proposition \ref{prop:ker}, where $G_{\mathbf{k}}^{\mathbf{p} } $ is   $\mathbb{C}^{k+1}$-valued and $C_{ \mathbf{q}}^{\mathbf{1} }  $ is scalar.
Now for any given $\sigma\in C_0^\infty (\mathbb{R}^3\setminus \{0\})$ and $\mathbf{q}\in \mathbb{Z}_+^2$, it is easy to see that $\psi$ as the inverse partial Fourier transformation of $\widetilde{\psi}_\tau$ given by
\begin{equation*}
   \widetilde{\psi}_\tau=\sigma(\tau) \widetilde{\mathscr  L}_{ \mathbf{0}      }^{(\mathbf{q}-\mathbf{ 1})} \otimes e_{1}^{{\tau}}\in  \ker \widetilde{\Box}_{b}
\end{equation*}
 by Proposition \ref{prop:ker}, is in $  L^2(\mathcal{N}) $  and satisfies $\psi\in \ker \Box_{b}$.
  Then $ \langle  g, \psi\rangle_{L^2(\mathbb{R}^7)} =  \langle Pg, \psi\rangle_{L^2(\mathbb{R}^7)} $ implies that
 \begin{equation*}
    \int_{\mathbb{R}^3} \langle \widetilde{ g}_\tau, \widetilde{\psi}_\tau \rangle_{L^2(\mathbb{R}^4)} d\tau=\int_{\mathbb{R}^3} \langle (\widetilde{Pg})_\tau, \widetilde{\psi}_\tau\ \rangle_{L^2(\mathbb{R}^4)} d\tau,
 \end{equation*}
   and so
\begin{equation*}\begin{split}\int_{\mathbb{R}^3}\left\langle G_{\mathbf{q}}^{\mathbf{1} }(\tau) ,e_{1}^{{\tau}} \right\rangle\sigma(\tau)\left(\frac 2\pi\right)^2|\tau|^2d\tau
=\int_{\mathbb{R}^3}C_{ \mathbf{q}}^{\mathbf{1} } (\tau)\sigma(\tau)\left(\frac 2\pi\right)^2|\tau|^2d\tau,
\end{split}\end{equation*}by using Lemma \ref{lem:L2}.
It follows  that $C_{ \mathbf{q}}^{\mathbf{1} }(\tau)=\left\langle G_{\mathbf{q}}^{\mathbf{1} }(\tau) ,e_{1}^{{\tau}} \right\rangle $ a.e. because  $\sigma(\tau)$ is an arbitrary $C_0^\infty(\mathbb{R}^3\setminus \{0\}) $ function.
Substitute it into (\ref{eq:pg-tau}) to get
\begin{equation*}\begin{split}
   ( \widetilde{Pg})_\tau&=\sum_{\mathbf{ q} }  \left\langle G_{\mathbf{q}}^{\mathbf{1} }(\tau) ,e_{1}^{{\tau}} \right\rangle \widetilde{ \mathscr L}_{\mathbf{0}  }^{(\mathbf{q }-\mathbf{ 1} )}\otimes e_{1}^{{\tau}} =P^\tau\left(\sum_{\mathbf{ q}   } G_{\mathbf{q}}^{\mathbf{1} }(\tau)\widetilde{ \mathscr L}_{\mathbf{0}  }^{(\mathbf{q }-\mathbf{ 1} )}   \right)\\&
   =P^\tau\left( \widetilde{g}_\tau*_\tau \widetilde{\mathscr  L}_{ \mathbf{0}      }^{(\mathbf{0})} \right)
  =  P^\tau\left(\widetilde{g}_\tau\right)*_\tau \widetilde{\mathscr  L}_{ \mathbf{0}      }^{(\mathbf{0})}  .
\end{split}\end{equation*}Here we use the property of the projection $*_\tau \widetilde{\mathscr  L}_{ \mathbf{0}      }^{(\mathbf{0})} $ in (\ref{eq:L-k-1-0}). By definition,  $\widetilde{\mathscr  L}_{ \mathbf{0}      }^{(\mathbf{0})}(  y,\tau)=\left(\frac {2 |\tau|}{\pi} \right)^2 e^{-|\tau||y|^2}$.
Thus for $g\in C_0^\infty((\mathbb{R}^7\setminus (\{0\})\times \mathbb{R}^3))$,  we get
\begin{equation*}\begin{split}
   Pg (x,t)&=  \frac{1}{(2\pi)^{3}}  \int_{\mathbb{R}^3} e^{it\tau} P^\tau\left(\widetilde{g}_\tau\right)*_\tau \widetilde{\mathscr  L}_{ \mathbf{0}      }^{(\mathbf{0})} (x) |\tau|^{2}d\tau \\&
 =  \frac{16 }{(2\pi)^{5}} \int_{\mathscr H}dyds\int_{\mathbb{R}^3} P^\tau g (x-y,s)e^{i (t-s) \tau- 2 i|\tau|B^{ \dot{{ \tau}}}(x, y)}   e^{-|\tau|  |  y|^2 } |\tau|^{ 2} d\tau
 \\&
 =  \frac{16 }{(2\pi)^{5}} \int_{\mathscr H} dyds\int_{S^2} d\dot{\tau} P^{\dot{\tau} }  g (y,s)\int_0^\infty r^4dr e^{-[i (-t+s) \dot{\tau}+2 i B^{\dot{ \tau}}(-y,x)+  | -y+ x |^2 ]r}\\&
 =g*S(x,t).
\end{split}\end{equation*}
with
\begin{equation*}
S(y,s):=  \frac{16\Gamma(5)}{(2\pi)^{5}} \int_{S^2}  \frac {P^{ {\tau} }}{ (     |y|^2 -i  s   \cdot {\tau} )^{  5}}d {\tau}, \qquad {\it for }\quad y\neq 0 .
\end{equation*}


\begin{thebibliography}{20}


\bibitem{BGV} {\sc Beals, R.,   Greiner, P. and  Vauthier J.},
The Laguerre Calculus on the Heisenberg Group, in {\it Special Functions: Group Theoretical Aspects and Applications} (R. Askey et
al. eds), 189-216,  Mathematics and Its Applications {\bf  18} Springer,  Dordrecht, 1984.

\bibitem{BGGV} {\sc Beals, R.,  Gaveau, B.,  Greiner, P. and  Vauthier J.},
The Laguerre calculus  on the Heisenberg group  II,
   {\it  Bull.   Sci. Math.},  \textbf{ 110 (3)},   225-288, (1986).

\bibitem{BG} {\sc Beals, R.  and  Greiner, P.},   Approximate identities from Laguerre functions and singular integrals on the Heisenberg group,  {\it  J. Anal. Math.}, \textbf{89}, 213-237,  (2003)

\bibitem{BR} {\sc Benedetti,  R. and  Risler, J.}, {\it Real  algebraic and semi-algebraic sets},  Actualit\'es math\'ematiques
Hermann,  Paris, 1990.

\bibitem{BCE}{\sc
Berenstein, C., Chang, D.-C.  and   Eby, J.}, $ L^p$ results for the Pompeiu problem with moments on the Heisenberg group,
 {\it   J.   Fourier Anal. and Appl.},  \textbf{ 10},   545-571,  (2004).


\bibitem{BCT}{\sc
Berenstein, C., Chang, D.-C.  and Tie, J.}, {\it Laguerre calculus and its applications on the Heisenberg group}, AMS/IP Studies in Advanced Mathematics {\bf  22},  American Mathematical Society, Providence, RI; International Press, Somerville, MA, 2001.



\bibitem{CCFI} {\sc  Calin, O. Chang, D.-C., Furutani, K. and Iwasaki, C.} {\it
Heat kernels for elliptic and sub-elliptic operators: Methods and techniques}, Birkh\"auser/Springer, New York, 2011.


\bibitem{CCT}{\sc
  Chang, D.-C., Chang S.-C. and Tie, J.},
Laguerre calculus and Paneitz operator on the Heisenberg group,
   {\it   Sci. in China  A-mathematics},
 \textbf{  52 (12)},  2549-2569, (2009).

 \bibitem{CGT}{\sc
  Chang, D.-C.,  Greiner, P.  and Tie, J.},
Laguerre expansion on the Heisenberg group and Fourier-Bessel transform on $\mathbb{C}^n$,
 {\it   Sci. in China  A-mathematics},
 \textbf{   49 (11)},   1722-1739,  (2006).

 \bibitem{CT}{\sc
  Chang, D.-C.  and Tie, J.},
 Estimates for spectral projection operators of the sub-Laplacian on the Heisenberg group
{\it   J.  Anal.  Math.},
 \textbf{ 71}, 315-347,  (1997).

 \bibitem{CG}{\sc
    Corwin, L.  and Greenleaf, F.}, {\it
Representations of Nilpotent Lie Groups and their Applications,  Part I. Basic Theory and Examples},  Cambridge Studies in Advanced Mathematics  \textbf{ 18}, Cambridge University Press, Cambridge, 1990.


\bibitem{Cox}
{\sc Cox, D.,  Little, J.  and  O'Shea, D.},
{\it Ideals, varieties, and algorithms.
An introduction to computational algebraic geometry and commutative algebra}, Third edition, Undergraduate Texts in Mathematics, Springer, New York, 2007.


\bibitem{Fral}{\sc
Fraleigh, J.},
{\it A first course in abstract algebra}, $7$-th edition, Pearson Education Limited, Pearson New International Edition,   2003.


\bibitem  {FS} {\sc Folland, G.B. and Stein, E.M.},  Estimates for the $\bar\partial_b$ complex and analysis on the Heisenberg group,
{\it Comm. Pure Appl. Math.}  {\bf 27}, 429-522, (1974).

\bibitem  {Ge} {\sc Geller, D.},  Fourier analysis on the Heisenberg group, {\it Pro. Natl. Acad. Sci. USA}, {\bf 74}, 1328-1331, (1977).

\bibitem  {G} {\sc Greiner, P.},  On the Laguerre calculus of left-invariant convolution (pseudodifferential) operators on the Heisenberg group, {\it Goulaouic-Meyer-Schwartz Seminar}, 1980-1981, Exp. {\bf  11}, Ecole Polytech., Palaiseau, 1981.

\bibitem  {J} {\sc Jerison, D.},  The Dirichlet problem for Kohn Laplacian on the Heisenberg group: I and II, {\it J. Func. Analysis}, {\bf 43}, 97-142, 224-257, (1981).


\bibitem{Katsumi}{\sc
Katsumi N.},  Characteristic roots and vectors of a diifferentiable
family of symmetric matrices, {\it
Linear and Multilinear Algebra} \textbf{1},  159-162 (1973).

\bibitem{Mik}{\sc Mikhlin, S.G.},  {\it Multidimensional singular integrals and integral equation}, International Series of Monographs in Pure and Applied Mathematics, \textbf{83}, Pergamon Press, Oxford-London-Edingburgh,  1963.


\bibitem{Mi}{\sc Mitrea, D.}, {\it Distributions, partial differential equations, and harmonic analysis}, Springer Verlag,  New York Heidelberg Dordrecht London 2013.

\bibitem{MM} {\sc Michele, L.  and  Mauceri, G.},   $L^p$ multipliers on the Heisenberg group,  {\it  Michigan Math. J}., \textbf{26}, 361-371, (1979).

\bibitem  {N} {\sc Nachman, A.},  The wave equation on the Heisenberg group, {\it Comm. PDEs}, {\bf 7}, 675-714, (1982).

\bibitem{Pe}{\sc Peetre, J.}, The Weyl transform and Laguerre polynomials, {\it Matematiche (Catania)\/},
{\bf 27}, 301-323, (1972).

\bibitem{PR}{\sc
Peloso, M. and Ricci, F.}, Analysis of the Kohn Laplacian on quadratic CR manifolds,  {\it J. Funct. Anal.}, {\bf 203}(2), 321-355,
(2003).

\bibitem{SW} {\sc Shi, Y. and
Wang, W.},
The Szeg\"o kernel for $k$-CF functions on the quaternionic
Heisenberg group, {\it Appl.  Anal.},  {\bf 96}, 2474-2492, (2017).

\bibitem{S} {\sc Strichartz, R.},
$ L^p$ Harmonic analysis and Radon transforms
on the Heisenberg group, {\it J. Func.  Anal.},
 \textbf{ 96},  350-406,  (1991).

\bibitem{T1} {\sc   Tie£¬ J.}, Imbedding $\mathbb{C}^1$ into $H_1$, {\it Canadian J. Math.}, {\bf 47}, 1317-1328, (1995).

\bibitem{T2} {\sc   Tie£¬ J.}, The explicit solution of the $\bar\partial$-Neumann problem in a non-isotropic Siegel domain, {\it Canadian J. Math.}, {\bf 49},  1299-1322, (1997).

\bibitem{ww} {\sc Wang, H. M. and
Wang, W.},
  On octonionic regular functions and Szeg\"o projection on the
Octonionic Heisenberg group, {\it Complex Anal. Oper. Theory},  {\bf 8},   1285-1324, (2014).

\bibitem{wang5} {\sc
Wang, W.},
 {  On the tangential Cauchy-Fueter operators on nondegenerate
quadratic hypersurfaces in $\mathbb{H}^{2}$,}
  {\it Math. Nachr.}, \textbf{286},  1353-1376,  (2013).


\end{thebibliography}
\end{document}